\documentclass[12pt]{amsart}
\usepackage{amsthm,amsopn,amsfonts,amssymb,pb-diagram}
\usepackage{hyperref,cancel}
\usepackage{fancyhdr}
\usepackage[dvips]{graphics}
\usepackage{young}
\usepackage{youngtab}

\usepackage{euler, amsfonts, amssymb, latexsym, epsfig, epic}
\usepackage{epstopdf}

\usepackage[OT2,OT1]{fontenc}
\newcommand\cyr{%
  \renewcommand\rmdefault{wncyr}%
  \renewcommand\sfdefault{wncyss}%
  \renewcommand\encodingdefault{OT2}%
  \normalfont
  \selectfont}
\DeclareTextFontCommand{\textcyr}{\cyr}

\newcommand\sh{\textcyr{X}}
\newcommand\sw{\Psi}
\newcommand\shij{\sh_{i\to j}}
\newcommand\shji{\sh_{j\to i}}
\newcommand\swij{\sw_{i\to j}}

\newtheorem{Theorem}{Theorem}[section]

\newtheorem*{Theorem*}{Theorem}
\newtheorem{Lemma}[Theorem]{Lemma}

\newtheorem{Proposition}[Theorem]{Proposition}

\theoremstyle{definition}

\theoremstyle{remark}

\newcommand{\onto}{\twoheadrightarrow}

\newcommand{\into}{\hookrightarrow}

\newcommand{\ZZ}{\mathbb{Z}}
\newcommand{\PP}{\mathbb{P}}

\newcommand\tensor{\otimes}
\newcommand\nwsearrow{\nwarrow \hskip -.16in \searrow}

\newcommand\defn[1]{{\bf #1}}

\newcommand\iso{\cong}

\newcommand\sweep{\Psi}

\newcommand\dom{\,\backslash\,}

\newcommand\integers{\ZZ}

\newcommand\naturals{{\mathbb N}}

\newcommand\union{\cup}
\newcommand\Union{\bigcup}

\newcommand\complexes{{\mathbb C}}
\newcommand\comment[1]{{\bf *** #1 *** }}
\newcommand\Gr[2]{Gr_{#1}(\AA^{#2})}
\newcommand\Grkn{\Gr{k}{n}}
\newcommand\wt[1]{{\widetilde #1}}
\newcommand\Gm{{{\mathbb G}_m}}

\newcommand\rank{{\rm rank}}
\newcommand\diag{{\rm diag}}
\newcommand\rowspan{{\rm rowspan}}

\renewcommand\AA{{\mathbb A}}
\newcommand\junk[1]{}

\begin{document}

\title{Schubert calculus and shifting of interval positroid varieties}
\author{Allen Knutson}
\address{Department of Mathematics, Cornell University, Ithaca, NY 14853 USA}
\email{allenk@math.cornell.edu}
\thanks{AK was partially supported by NSF grant DMS-0902296.}
\date{\today. DRAFT}

\begin{abstract}
  Consider $k\times n$ matrices with rank conditions placed on intervals of 
  columns. The ranks that are actually achievable correspond naturally to upper 
  triangular partial permutation matrices, and we call the corresponding 
  subvarieties of $\Grkn$ the {\em interval positroid varieties}, as this
  class lies within the class of positroid varieties studied in
  [Knutson-Lam-Speyer].  It includes Schubert and opposite
  Schubert varieties, and their intersections. \junk{Some other subclasses
  will be studied in a companion paper [Knutson-Lederer].}

  Vakil's ``geometric Littlewood-Richardson rule''
  [Vakil] uses certain degenerations to positively compute the
  $H^*$-classes of Richardson varieties, each summand recorded
  as a $(2+1)$-dimensional ``checker game''.
  We use his same degenerations to positively compute
  the $K_T$-classes of interval positroid varieties,
  each summand recorded more succinctly
  as a $2$-dimensional ``$K$-IP pipe dream''.
  \junk{
    , and show
    that the problem is essentially equivalent to Kogan Schubert calculus
    (for which an algorithm, but not really a formula, was given 
    in [Knutson-Yong '05???]).
  }
  In Vakil's restricted situation these IP pipe dreams biject very simply
  to the puzzles of [Knutson-Tao].

  \junk{
  We introduce ``positroid schemes'', which are intersections
  (reduced but typically reducible) of positroid varieties, and classify
  them in terms of alternating sign matrices. We give an algorithm to
  decompose them.
  }

  We relate Vakil's degenerations to Erd\H os-Ko-Rado shifting, and
  include results about computing ``geometric shifts'' of general
  $T$-invariant subvarieties of Grassmannians.
\end{abstract}

\maketitle

\setcounter{tocdepth}{1}
{\footnotesize 
\tableofcontents}


\section{Introduction, and statement of results}

\subsection{Interval positroid varieties}\label{ssec:introIP}

Define the following \defn{interval rank function} $r$, from $k\times n$ 
matrices over a field, to the space of upper-triangular $n\times n$ matrices:
$$ M \mapsto r(M), \qquad r(M)_{ij} := \rank(
\text{the submatrix of $M$ using columns $\{i,i+1,\ldots,j\}$}) 
$$
Note that $r$ is unchanged by row operations, so is only a function of
the row span, and hence descends to a function on the $k$-Grassmannian
$\Grkn$.

It turns out (proposition \ref{prop:intdefsagree}) that the data of 
$r(M)$ is equivalent to that of an upper triangular
partial\footnote{meaning, at most one $1$ in any row and column} 
permutation matrix $f(M)$ of rank $n-k$, where
$$ r(M)_{ij} 
= \big|[i,j]\big| 
- \#\{\text{$1$s in $f(M)$ that are weakly Southwest of $(i,j)$}\}.
$$
Conversely, given the partial permutation $f$ (and its associated
rank matrix $r$) we can define two \defn{interval positroid varieties} 
in the $k$-Grassmannian:
\begin{eqnarray*}
  \Pi_f^\circ &:=& \{\rowspan(M) : M \in M_{k\times n},\ \rank(M)= k,\ r(M) = r \}\\
  \Pi_f &:=& \{\rowspan(M) : M \in M_{k\times n},\ \rank(M)= k,\
  r(M) \leq r \text{ entrywise} \} 
\end{eqnarray*}
By proposition \ref{prop:intdefsagree},
these are special cases of the {\em positroid varieties} studied in \cite{KLS},
giving us the facts that
\begin{enumerate}
\item $\Pi_f^\circ$ is smooth and irreducible (in particular, nonempty),
  and $\Pi_f$ is its closure,
\item $\Pi_f$ is normal and Cohen-Macaulay, with rational singularities, and
\item the intersection of any set of $\{\Pi_f\}$ is a (reduced) union 
  of others.
\end{enumerate}
More specifically, they are the Grassmann duals of the {\em projection
  varieties} of \cite{BC}, which are not as general as the
{\em projected Richardson varieties} of \cite{KLS2}, which are
(in type $A$) exactly all the positroid varieties. 
I thank Brendan Pawlowski for help navigating this terminology.

If the partially defined $f$ is defined exactly on $[k+1,n]$, and increasing
on there, then $\Pi_f$ is a \defn{Schubert variety}. 
The class of interval positroid varieties also includes
opposite Schubert varieties (by reversing the interval), and
their intersections, the \defn{Richardson varieties.} 
Still more generally, it includes 
(theorem \ref{thm:vakilvarieties}) the varieties appearing in 
Vakil's paper \cite{Vakil} used to compute Schubert calculus on $\Grkn$.

In this paper we answer the following questions (really, one question):
\begin{quotation}
  What is the expansion of the cohomology class, or better, 
  the 
  equivariant $K$-theory class $[\Pi_f]$
  in the opposite Schubert basis $\{ [X^\lambda] \}$ of $K_T(Gr_k(\AA^n))$?

  These coefficients are known to be positive in a suitable sense
  \cite{AGM}; what is a combinatorial formula for which this
  positivity is manifest?
\end{quotation}

An answer to the first question was given in \cite{KLS} 
(in $H^*_T$) and \cite{HeLam} (in $K_T$), in terms of affine Stanley
symmetric functions, but it is not manifestly positive.

As our results will look exactly the same in equivariant cohomology
(over $\complexes$) as in the equivariant Chow ring (over an arbitrary field),
and in topological vs. algebraic $K$-theory, we will use the 
more-familiar topological terminology throughout.

In \S \ref{ssec:HT} we state our formul\ae\ in ordinary and
equivariant cohomology. 
In \S \ref{ssec:shifting} we describe the geometry we use to derive
this formula, an extension of the degenerative technique from \cite{Vakil}.
In \S \ref{ssec:main} we give the actual derivation.
In \S \ref{ssec:KT} we explain the modifications necessary to 
compute in (equivariant) $K$-theory.

In particular, when $\Pi_f$ is a Richardson variety this allows us to
extend Vakil's results from cohomology to equivariant $K$-theory. 
In a companion paper \cite{KL} we apply these results to ``direct sums
of Schubert varieties'', another class of interval positroid varieties.

\subsection{IP pipe dreams}\label{ssec:HT}

Consider the \defn{label set} $\{0,1\} \,\cup\, \{A,B,\ldots\}$,
where only the latter group are called \defn{letters}, and consider
the following tile schema, with \defn{pipes} connecting the edges of a square:

\centerline{\epsfig{file=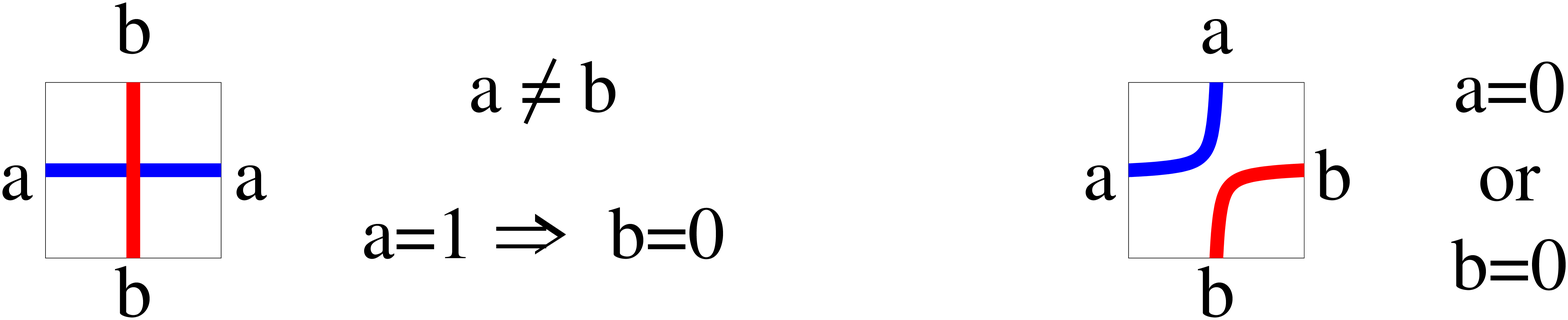,height=1.5in}}

Call these the \defn{crossing} and \defn{elbows} tiles,
and the $a=b=0$ elbows the \defn{equivariant tile}\footnote{%
  The $0$s and $1$s on these tiles are not quite the same as
  those on the puzzle pieces from \cite{KT}; see theorem
  \ref{thm:vakilvarieties} for the connection.}.
We will often want to determine a tile from its South and East labels,
and this can be done uniquely unless both are $0$.

We will tile these together, such that the boundary labels of 
adjoining tiles match up, making continuous ``pipes'' from boundary
to boundary bearing well-defined labels. 
Define an \defn{IP pipe dream} (the IP for ``interval positroid'') to
be a filling of the upper triangle of an $n\times n$ matrix, such that

\begin{center}
  {
    \begin{minipage}[t]{0.52\linewidth}
\begin{itemize}
\item on the East edges (of each $(i,n)$ square), there are no $1$ labels,
\item on the South edges (below each $(i,i)$ square), there are no $0$ labels,
\item on the West edges (West of each $(i,i)$ square), 
  there are {\em only} $0$ labels\\ (we will derive this from other\\
  conditions, in proposition \ref{prop:west0s}),
\item on the North edges (above each $(i,1)$ square),
  there are only $0$s and $1$s,
\item no two pipes of the same label cross, and finally,
\item {\bf no two {\em lettered} pipes cross twice.} \\
  This is the only nonlocal condition.
\end{itemize}

    \end{minipage} }
  \hfill
  \raisebox{-2.3in}
  {
    \begin{minipage}[t]{0.42\linewidth}
      {\epsfig{file=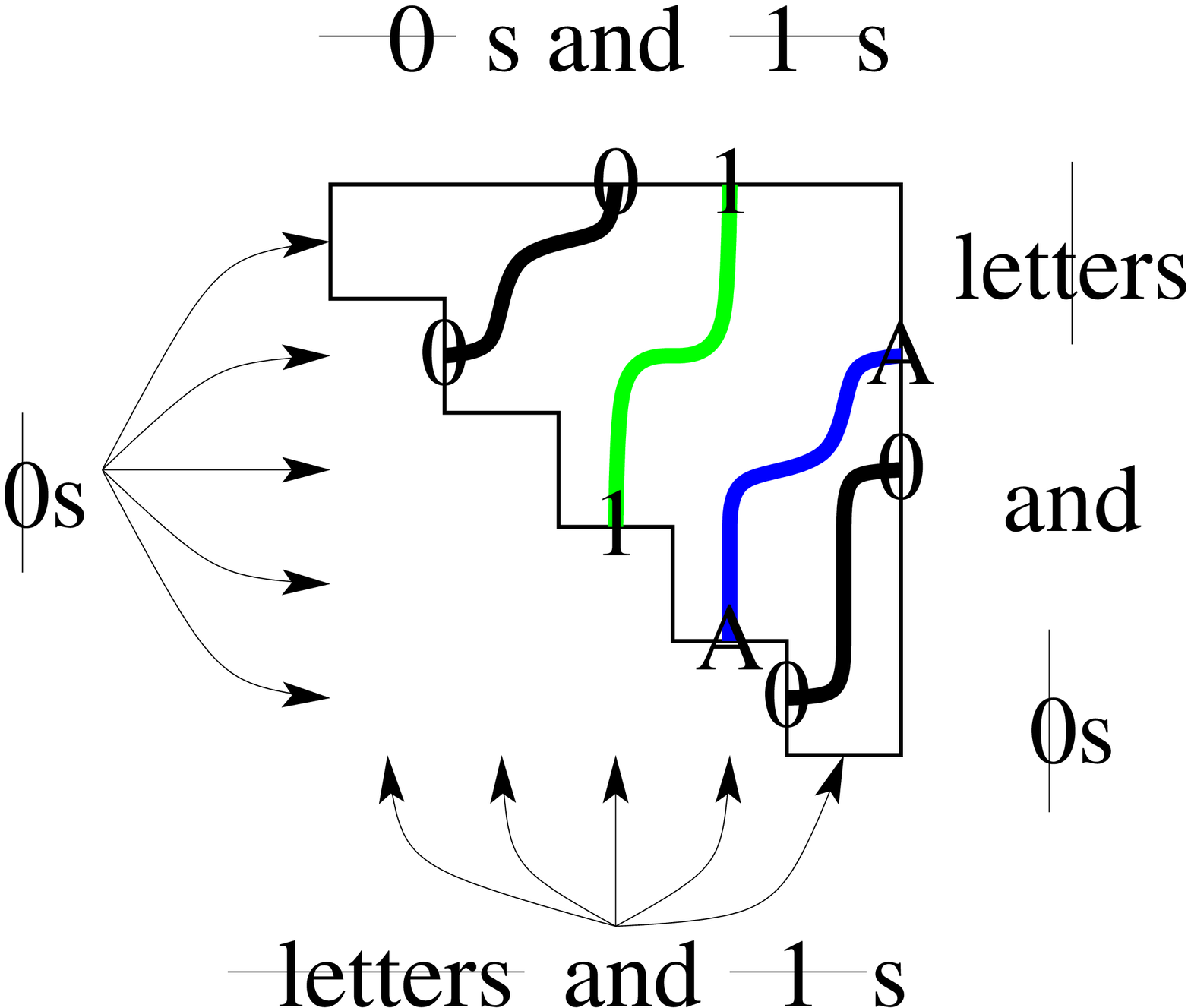,height=2.3in}}
    \end{minipage} }
\end{center}

In fact the $1$ acts more like a special letter than like the $0$
(especially in \S \ref{ssec:viable});
for example the ``no two lettered pipes cross twice'' rule applies even
if $1$ is considered to be a letter, because of the second condition
on crossing tiles.

\begin{figure}[htbp]
  \centering
  \epsfig{file=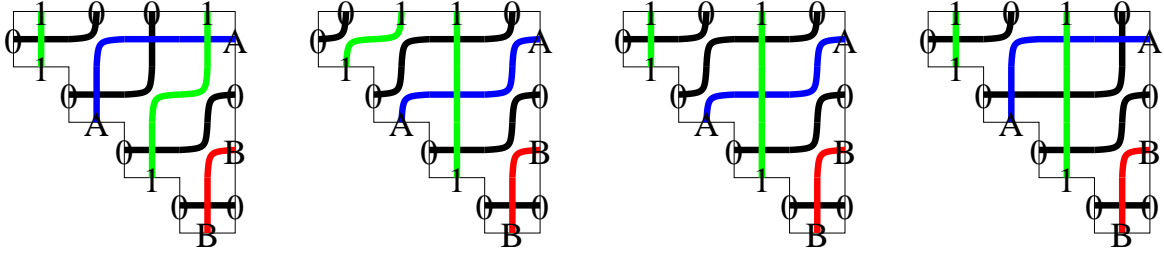,height=1.3in}
  \caption{The IP pipe dreams
    whose partial permutation is $1\mapsto 2,3\mapsto 4$. 
    (The lettered pipes connect the $1$st and $3$rd East edges 
    to the $2$nd and $4$th South edges.)
    Note that in the first and fourth figures, the $A$ pipe crosses a $0$ pipe 
    twice, but that's permissible because the $0$ isn't a lettered pipe.  }
  \label{fig:IPpipesEx}
\end{figure}

\newcommand\hlabel[1]{-\!\!\!-\!\!\!\!\!\!\!\! #1\ }
\newcommand\hlabelzero{-\!\!\!-\!\!\!\!\!\! 0\ }
\newcommand\hlabelone{-\!\!\!\!-\!\!\!\!\!\! 1\ }
\newcommand\vlabel[1]{\,\big|\!\!\! #1}

Each lettered pipe connects a horizontal edge below the diagonal to
a vertical edge on the East side. Since $a\neq b$ in crossing tiles,
the $i$th $\hlabel B\, $ from the left must connect to
the $i$th $\vlabel B$ from the top. 
By the nonlocal condition, the $i$th $A$ pipe will cross the $j$th $B$ pipe
either once or not at all, and can be predicted from the boundary
and the Jordan curve theorem.

We think of two IP pipe dreams as equivalent if they differ only in
the letter labels. This {\em includes} the possibility of folding two
letters into the same letter (only allowed if those pipes don't cross,
which as just explained can be predicted from the boundary).

To an IP pipe dream $P$, we associate two objects: 
\begin{itemize}
\item $f(P)$, an upper triangular partial permutation
  depending on only\\ the South and East labels of $P$, and 
\item $\lambda(P)$, a partition depending on only the North labels.
\end{itemize}
The partial permutation $f(P)$ is induced by the lettered pipes, 
(i.e. not labeled $0,1$), as follows. For each lettered pipe in $P$,
place a $1$ in $f(P)$ above the South end
of the pipe, and left of the East end. In particular, the $1$s coming 
from pipes of a given letter are arranged NW/SE (since such pipes
don't cross).
The IP pipe dreams in figure \ref{fig:IPpipesEx} are all those with 
$f(P)$ being the partial permutation $1\mapsto 2, 3\mapsto 4$.

The English partition $\lambda(P)$ in the fourth quadrant of
the Cartesian plane is read as follows. Start at the point
$(0,\#0$s across the North side$)$  and reading the North side of $P$
from left to right, move down for each $1$, and left for each $0$.
The region above the resulting path is the partition $\lambda(P)$,
and $\dim X^{\lambda(P)} = |\lambda(P)|$.
As we will see later, $|\lambda(P)| + |\mu| = |\nu| + \#$equivariant tiles
in $P$. In the IP pipe dreams in figure \ref{fig:IPpipesEx}, the partitions
are $(2) = {\tiny\yng(2)}, (1,1) = {\tiny\yng(1,1)}, 
(1) = {\tiny\yng(1)}, (1) = {\tiny\yng(1)}$ respectively.

\begin{Theorem}\label{thm:Hformula}
  In $H^*_T(\Grkn)$, expanding $[\Pi_f]$ in the
  $\integers$-basis of opposite Schubert classes gives
  $$ [\Pi_f] 
  = \sum_{P: f(P)=f\atop \text{ $P$ has no equivariant tiles}} [X^{\lambda(P)}]
  \quad=\quad \sum_\lambda \#\bigg\{ P\ :\ {f(P)=f,\ \lambda(P)=\lambda \atop
  \text{ $P$ has no equivariant tiles} } \bigg\}\ [X^\lambda].
  $$
\end{Theorem}

Let $T \leq GL(n)$ be the diagonal matrices. As $\Pi_f$ is preserved by 
this group, it defines a class in $H^*_T(\Grkn)$, again denoted $[\Pi_f]$. 
The corresponding expansion in the basis requires coefficients
from $H^*_T(pt) \iso Sym(T^*) \iso \integers[y_1, \ldots, y_n]$, where $y_i$
is the character $y_i(\diag(t_1,\ldots,t_n)) = t_i$ on $T$.

Define $wt(P) \in H^*_T(pt)$ (for ``weight'') as the 
product of $y_{row(t)}-y_{col(t)}$, over all equivariant tiles $t$.
In the IP pipe dreams in figure \ref{fig:IPpipesEx}, the weights are
$1, 1, y_1-y_2, y_2-y_4$ respectively.

\begin{Theorem}\label{thm:HTformula}
  In $H^*_T(\Grkn)$, expanding $[\Pi_f]$ in the
  $H^*_T(pt)$-basis of opposite Schubert classes gives
  $$ [\Pi_f] = \sum_{P:\ f(P) = f} wt(P)\ [X^{\lambda(P)}]. $$
  Specializing each $y_i$ to $0$ recovers the previous theorem.
\end{Theorem}

This formula is manifestly Graham-positive\footnote{%
  Moreover, Graham's derivation shows that if $X \subseteq G/P$ is a
  subvariety, and $[X] = \sum_\pi c_\pi [X^\pi], c_\pi \in H^*_T$ is the 
  expansion in opposite Schubert classes, then each coefficient $c_\pi$ 
  is not only a sum of products of simple roots, but can be written as a
  sum of products of {\em distinct, positive} roots.
  This formula for $[\Pi_f]$ also does this.}
\cite{Graham}. In the figure \ref{fig:IPpipesEx} example, it says
$$ [\Pi_{1\mapsto 2,\ 3\mapsto 4}] = [X^{(2)}] + [X^{(1,1)}] 
+ (y_1-\cancel{y_2}+\cancel{y_2}-y_4) [X^{(2,1)}]
\qquad \in H^*_T(\Grkn). $$

\subsection{Shifting and sweeping}\label{ssec:shifting}

In this paper ``variety'' means ``reduced and irreducible scheme'',
and any ``subvariety'' will be closed. Also, $[n]$ denotes $\{1,2,\ldots,n\}$.

Let $X$ be a $T$-invariant subvariety of $\Grkn$, 
and for $i,j \in [n]$ define the \defn{(geometric) shift} of $X$
$$ \sh_{i\to j} X := \lim_{t\to \infty} \exp(t e_{ij}) \cdot X $$
where $e_{ij}$ is the matrix with a $1$ at $(i,j)$ and $0$s elsewhere.
(The precise definition of such a limit is recalled in \S \ref{sec:shifting}.)
Then the limit scheme $\sh_{i\to j} X$ is again $T$-invariant, and
defines the same homology and $K$-class as $X$ itself.
I learned of this construction from \cite{Vakil}, but as we
explain in \S \ref{sec:shifting}, it is closely related to the
Erd\H os-Ko-Rado shifting construction \cite{EKR} in extremal combinatorics.

To keep track of the equivariant class, we also need the 
\defn{(geometric) sweep} of $X$, 
$$ \sweep_{i\to j} X := \overline{\bigcup_{t\in \AA^1} \ \exp(t e_{ij}) \cdot X} $$

For general $X$, these schemes can be very difficult to compute; we give
some general results in \S \ref{sec:shifting}. But certain shifts of
certain interval positroid varieties are tractable.

Call $(i,j)$ an \defn{essential box} for the partial permutation $f$
if its rank condition $r(M)_{ij} \leq r(f)_{ij}$ is not implied by the
rank condition for any of $(i\pm 1,j),(i,j\pm 1)$. (There is an easy
combinatorial description of these from \cite{Fulton92}, 
recalled in \S \ref{sec:IP}.) Call a shift $(i,j)$ \defn{safe for $f$} if
for each essential box $(i',j')$, either
 $i \in [i',j']$, or  $j\notin [i',j']$, or  $(i',j') = (i+1,j)$.

\begin{Theorem}\label{thm:safeshiftintro}
  If the shift $(i,j)$ is safe for $f$, 
  then $\sw_{i\to j} \Pi_f$ is again an interval positroid variety,
  and $\sh_{i\to j} \Pi_f$ is a certain reduced union $\Union_{f' \in C} \Pi_{f'}$ 
  of interval positroid varieties.
  If $(i+1,j)$ is indeed an essential box for $f$, then
  $$ [\Pi_f] = (y_i-y_j) \big[\sw_{i\to j} \Pi_f\big] + \sum_{f'\in C} [\Pi_{f'}] $$
  as elements of $H^*_T(\Grkn)$. 
  If $(i+1,j)$ is not an essential box for $f$, then
  $\Pi_f$ is \defn{$\sh_{i\to j}$-invariant}, 
  in that $\sh_{i\to j} \Pi_f = \sw_{i\to j} \Pi_f = \Pi_f$.

  If $f = f(P)$ for an IP pipe dream using $m$ distinct letters, then
  $\sh_{i\to j} \Pi_f$ has at most $m+1$ components.\footnote{I thank
    Mathias Lederer for this observation.
    The important case $m=1$ is explored in \S \ref{sec:puzzles}.} 
  The intersection of any set 
  $S\subseteq C$ of these components is again an interval positroid variety
  $\Pi_{f(S)}$. In particular, as $K$-classes,
  $$ [\Pi_f] = [\sh_{i\to j} \Pi_f] 
  = \bigg[ \Union_{f'\in C} \Pi_{f'} \bigg]
  = \sum_{S \subseteq C,\ S\neq \emptyset} (-1)^{|S|-1}\ [\Pi_{f(S)}].
  $$
\end{Theorem}

The precise version of the theorem (enumerating the components in $C$)
will be theorem \ref{thm:safeshift}, 
which also includes the extension to $K_T(\Grkn)$.

\subsection{The Vakil sequence}\label{ssec:main}

Given an IP pipe dream $P$, and a box $(i, j), i\leq j$, 
define (as in figure \ref{fig:slice}) the
\defn{slice $s(P,i,j)$ of $P$ at $(i,j)$} as the data of the labels on
\begin{itemize}
\item  the South edges of 
  $\{ (m,m) : m < i \}
  \ \cup\ \{ (i,m) : i \leq m \leq j \}
  \ \cup\ \{ (i-1,m) : j < m \leq n \}$
\item the East edges of $(i,j)$ and $\{ (m,n) : m < i \}$.
\end{itemize}

\begin{figure}[htbp]
  \centering
  \epsfig{file=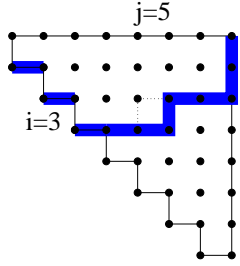,height=1.5in}
  \caption{An $n=6$ example of a slice at $(3,5)$.
    The slice is the thick blue edges and the (unpictured) labels thereon.}
  \label{fig:slice}
\end{figure}

Not every labeling $s$ of these edges arises from an IP pipe dream; 
we spell out the conditions in \S \ref{sec:main}. 
To each ``viable'' slice $s$, 
we associate in \S \ref{sec:main} a partial permutation $f(s)$.
For now, it suffices to mention that $f(s(P,n,n)) = f(P)$, 
and $\Pi_{f(s(P,1,1))} = X^{\lambda(P)}$.

Given a slice $s$, we can consider what tiles can be placed at $(i,j)$,
making new slices $s'$ at $(i,j-1)$ (or at $(i-1,n)$ if $j=i$).

\begin{Theorem}\label{thm:main}
  Let $s$ be a slice at $(i,j)$, and $\Pi_{f(s)}$ the associated 
  interval positroid variety (defined in \S \ref{sec:main}). 
  The shift $(i,j)$ is safe for $f(s)$.
  Let $\{s'\}$ be the set of viable slices arising from a tile at $(i,j)$.
  \begin{enumerate}
  \item If the South and East labels of $(i,j)$ in $s$ are not both zero,
    then $\Pi_{f(s)}$ is $\sh_{i\to j}$-invariant. There is a unique $s'$,
    and its $f(s')$ is unchanged from $f(s)$.
    \begin{enumerate}
    \item If the labels are equal but not $0$, forcing the 
      $a=0$ elbows tile, then $f(s)\big(i\big) = j$.
    \item If the labels are distinct, forcing the crossing tile,
      then $f(s)\big(i\big) \neq j$ (and is undefined if the East label is $0$).
    \end{enumerate}
  \item If the South and East labels of $(i,j)$ in $s$ are both $0$,
    then the various $\{\Pi_{f(s')}\}$ are the sweep (for the 
    equivariant tile) and the components of the shift (for the other possible
    tiles).
  \end{enumerate}
\end{Theorem}

This is the inductive step by which theorem \ref{thm:HTformula} is proved,
where $i$ decreases from $n$ to $1$,
and for each $i$ we take $j$ from $n$ down to $i+1$.
This is exactly the sequence of shifts used in \cite[\S2.2]{Vakil},
though Vakil doesn't use the shifting formalism.

\subsection{Extension to $K$-theory}\label{ssec:KT}

\subsubsection{The $K$-tiles}\label{sssec:Ktiles}

To compute in equivariant $K$-theory, we need a new kind of label $W$ on the 
\emph{vertical} edges: it is a word in $\{1\}\cup \{A,B,\ldots\}$ (no $0$s),
no letters repeating, and if it contains $1$ then the $1$ must be at the end.
There are now four kinds of tiles, including the 
fundamentally new ``displacer'' tile:

\centerline{\epsfig{file=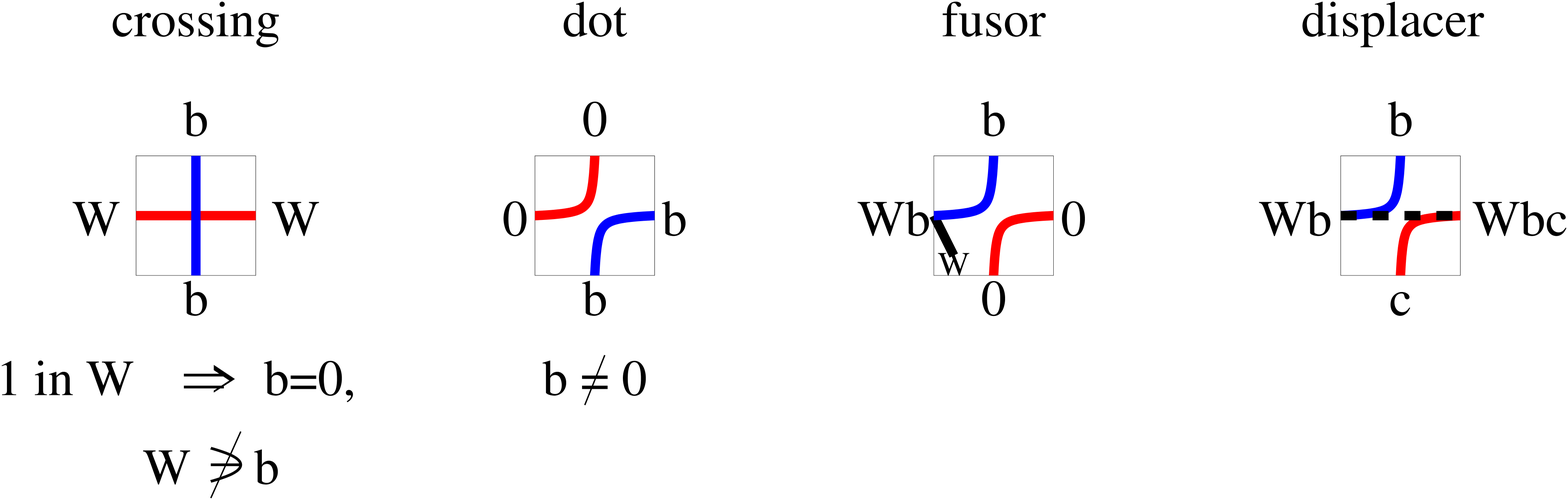,height=1.7in}}

Define a \defn{$K$-IP pipe dream} as one built from these tiles, with
the same conditions as on an IP pipe dream, plus one more nonlocal condition: 
{\bf two pipes appearing in the same word {\em must} cross once (and, of
course, not twice).} The meeting of two pipes in a fusor or displacer
tile doesn't count as a crossing. Note that IP pipe dreams are a subclass
of $K$-IP pipe dreams, where $|W|=1$ in the crossing tiles,
$|W|=0$ in the fusor tiles, and there are no displacer tiles.

\begin{figure}[htbp]
  \centering
  \epsfig{file=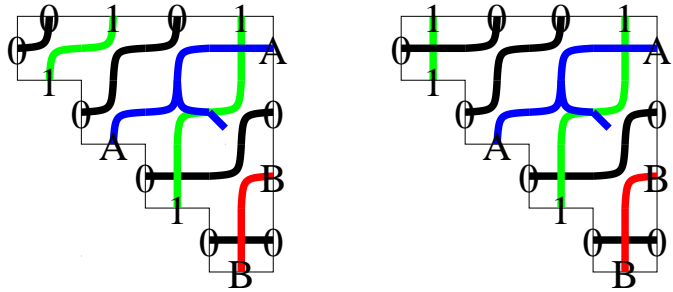,height=1.5in}
  \caption{The $K$-IP pipe dreams whose partial permutation is
    $1\mapsto 2,3\mapsto 4$ that didn't appear in figure
    \ref{fig:KIPpipesEx}. Each fuses an $A$ and a $1$ pipe,
    at $(2,4)$.}
  \label{fig:KIPpipesEx}
\end{figure}

Notice that if, on each edge of a $K$-IP pipe dream we erase every
label except the last one, we get a consistent system of unbroken pipes,
and missing labels can be reconstructed uniquely from the visible ones.
However, the nonlocal conditions that say which systems of pipes can be
extended to a $K$-IP pipe dream seem too complicated to be useful.


As theorem \ref{thm:safeshiftintro} suggests, there are signs in the
$K$-formula, but (as predicted by Brion's theorem \cite{BrionPos})
they are determined by the parity of the codimension. 
Let $fusing(P)$ denote the sum over the fusor tiles, of the size $|W|$ 
of their word. (So $fusing(P)=0$ iff the $K$-IP pipe dream $P$ is
an ordinary IP pipe dream, since the presence of a displacer tile 
forces the appearance of a fusor tile with $|W|>0$ to the East of it.)

\begin{Theorem}\label{thm:Kformula}
  In $K^*(\Grkn)$, expanding $[\Pi_f]$ in the
  $\integers$-basis of opposite Schubert classes gives
  \begin{eqnarray*}
    [\Pi_f] &=& \sum_{P:\ f(P) = f \atop   \text{ $P$ has no equivariant tiles} }
    (-1)^{fusing(P)} [X^{\lambda(P)}] \\
    &=& \sum_\lambda (-1)^{\dim \Pi_f - \dim X^\lambda}
    \#\bigg\{P\ :\ {f(P)=f,\ \lambda(P)=\lambda \atop
    \text{ $P$ has no equivariant tiles}} \bigg\}\ [X^\lambda].
  \end{eqnarray*}
  If $P$ is a $K$-IP pipe dream, 
  then $fusing(P) = \dim \Pi_{f(P)} -  \dim X^{\lambda(P)}$,
  so this formula is positive in the sense of \cite{BrionPos}.
\end{Theorem}

Combining the nonequivariant pipe dreams from 
figures \ref{fig:IPpipesEx} and \ref{fig:KIPpipesEx}, we get
$$ [\Pi_{1\mapsto 2,3\mapsto 4}] = [X^{(2)}] + [X^{(1,1)}] - [X^{(1)}]
\qquad \in K(\Grkn). $$
Notice that as compared to $H^*$, the extra terms in $H^*_T$ come from
larger varieties $X^\lambda$, whereas the extra terms in $K$ come from
smaller varieties.

\subsubsection{$T$-equivariance}

The base ring $K_T(pt)$ of $T$-equivariant $K$-theory is the Laurent
polynomial ring $Rep(T) \iso \integers[\exp(\pm y_1),\ldots,\exp(\pm y_n)]$, 
written thus for comparison to equivariant cohomology. 
Here $\exp(y_i)$ denotes the $K_T$-class of the one-dimensional
representation with character $y_i$.

Define the \defn{$K_T$-weight} $wt_K(P)$ of a $K$-IP pipe dream as
$$ wt_K(P) := \prod_{i<j}
\begin{cases}
  1 - exp(y_j-y_i) & \text{ if the tile at $(i,j)$ is the equivariant tile
  (all $0$s)}\\
  \phantom{0} + exp(y_j-y_i) & \text{ if the tile at $(i,j)$ has $0$s
    on its South and East only} \\
 1  \phantom{0 +1} & \text{ otherwise}.
\end{cases}
$$
The special role of the tiles with $0$ on the South and East becomes
clear in \S \ref{ssec:placing}.

\begin{Theorem}\label{thm:KTformula}
  In $K_T^*(\Grkn)$, the expansion of $[\Pi_f]$ in the
  $\integers[\exp(\pm y_1),\ldots,\exp(\pm y_n)]$-basis 
  of opposite Schubert classes is
  $$ [\Pi_f] = \sum_{P:\ f(P) = f} (-1)^{fusing(P)} wt_K(P)\ [X^{\lambda(P)}] $$
  which is positive in the sense predicted in \cite[Corollary 5.1]{AGM}.

  Specializing each $y_i=0$ recovers the previous theorem.
  Dropping the $fusing(P)>0$ summands, and taking the lowest-degree term
  in the $(y_i)$, recovers the $H^*_T$-formula from the $K_T$.
\end{Theorem}

For example, the $K_T$-weights of the $K$-IP pipe dreams from
figures \ref{fig:IPpipesEx} and \ref{fig:KIPpipesEx} are (in order)
$$ \exp(y_2-y_4),
\qquad \exp(y_1-y_2),
\qquad 1 - \exp(y_1-y_2),
\qquad 1 - \exp(y_2-y_4), $$
$$ \exp(y_1 - \cancel{y_2} + \cancel{y_2}-y_4),
\qquad (1-\exp(y_1-y_2))\exp(y_2-y_4) = \exp(y_2-y_4) - \exp(y_1-y_4). $$

\subsection{Outline of the paper}

In \S \ref{sec:IP} we recall the basic properties we need of
interval positroid varieties, and in particular define their essential 
and crucial boxes.
In \S \ref{sec:shifting} we give some results about geometric and 
combinatorial shifting, and prove theorem \ref{thm:safeshiftintro}
about safe shifts of positroid varieties.
In \S \ref{sec:main} we prove the main theorem, that $K$-IP pipe dreams serve 
as a record of the degeneration process defined by Vakil \cite{Vakil}, 
in enough detail to recover the $K_T$-class.  
In \S \ref{sec:puzzles} we connect IP pipe dreams to the
equivariant puzzles of \cite{KT}.

\junk{
In \S \ref{sec:Kogan} we show that computing the $K$-classes of interval
positroid varieties is the same problem as computing Kogan Schubert
calculus (in $K$-theory), for which we already gave an algorithm in \cite{KY}.
}
The combinatorial difference between the pipe dream calculus laid out here,
as contrasted with the checker games of \cite{Vakil},
is that an IP pipe dream serves as a $2$-dimensional record
of a $(2+1)$-dimensional checker game. 

\section*{Acknowledgments}

These ideas have been a long time brewing, and discussed fruitfully
with many people, in particular Nicolas Ford, David Speyer, Ravi Vakil, 
Alex Yong, and especially Mathias Lederer. Our deepest thanks go to
Franco Saliola for his virtuosic implementation of IP pipe dreams in 
{\tt sage}. 

We first described the connection between Vakil's degeneration and
Erd\H os-Ko-Rado shifting in the unpublished preprint \cite{unpub},
whose results are fully subsumed here.

\section{Interval positroid varieties}\label{sec:IP}

Most of the definitions in this section (though not the one in its title) 
are from \cite{KLS}.

\subsection{Positroid varieties and their covering relations}
\label{ssec:positroid}

A \defn{juggling pattern of length $n$} is a bijection
$J:\integers \to \integers$ such that $f(i)-i$ is periodic with period $n$,
and $f(i)-i \geq 0$ for all $i\in \integers$. 
The \defn{siteswap} of $J$ is the $n$-tuple $f(1)-1,\ldots,f(n)-n$, and 
obviously $J$ can be reconstructed from its siteswap.
The average $k$ of the siteswap turns out to be the number of orbits of $J$ that
aren't fixed points, called the \defn{ball number}. Hereafter fix $n$ and $k$.

Call $J$ \defn{bounded}\footnote{%
  While $f(i)\geq i$ is natural from the juggling point of view, 
  in that it says balls land after they are thrown, 
  the $f(i) \leq i+n$ condition is already violated by the
  standard $3$-ball pattern $n=1$, $f(i)=i+3\ \forall i$.}
if $f(i)-i \leq n$ for all $i\in \integers$.
For such $J$, define the following variety $\wt\Pi_J \subseteq M_{k\times n}$
by rank conditions on all {\em cyclic} intervals:
\begin{eqnarray*}
  r(J)_{ij} &:=& \big|[i,j]\big| 
  - \#\{\text{$1$s in $J$'s matrix that are weakly Southwest of $(i,j)$}\},
  \quad i \leq j \leq i+n \\
  \wt\Pi_J &:=& \{M :\ M \in M_{k\times n},\ 
  \rank(\text{columns }[i,j] \bmod n) \leq r(J)_{ij},\quad
  \forall i\in \integers, j\in [i,i+n] \}
\end{eqnarray*}
The \defn{positroid variety} $\Pi_J \subseteq \Grkn$ is defined as
$$ \Pi_J := \{\rowspan(M) : M \in \wt\Pi_J,\ \rank(M) = k\}. $$
All the properties claimed in \S \ref{ssec:introIP} of interval
positroid varieties are in fact true of positroid varieties, as proven
in \cite[\S 5.4--5.5]{KLS}.

\newcommand\ess{{\rm ess}}

We depict $J$ as an infinite, periodic permutation matrix, with \defn{dots}
in the boxes\footnote{%
  This may well be the transpose of the convention you are used to!
\junk{
  Because of this, our flag manifold later will be $B_- \dom GL(n)$,
  not $GL(n)/B$. \comment{Unless I remove the Kogan stuff}}
}
$(i,J(i))$, $i\in \integers$. To construct the \defn{diagram} of $J$,
we cross out all boxes strictly to the West or South (but not both)
of each dot, leaving the diagram as the remainder. The \defn{essential set}
$\ess(J)$ is the set of Northeast corners of the diagram (note that the diagram
has one unbounded component, stretching North and East to infinity
with no Northeast corner). It is not difficult to prove 
(in analogy with \cite{Fulton92}) that
$$   \wt\Pi_J = \{M :\ M \in M_{k\times n},\ 
  \rank(\text{columns }[i,j] \bmod n) \leq r(J)_{ij},\quad
  \forall (i,j) \in \ess(J)\}. $$

\begin{figure}[htbp]
  \centering
  \epsfig{file=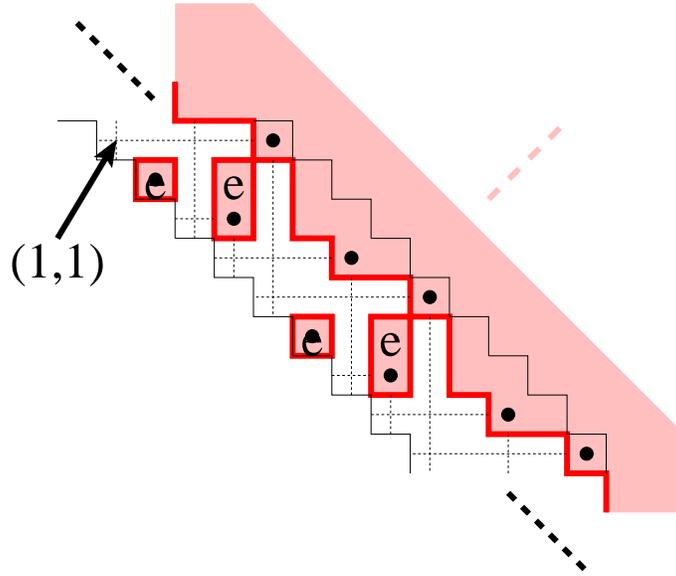,height=3in}
  \caption{The diagram of the juggling pattern with siteswap $4013$,
    in pink.
    Up to $4$-periodicity, it has two essential boxes, at $(2,2)$ and $(2,4)$,
    whose associated rank conditions are $\rank [$column $2] \leq 0$,
    $\rank [$columns $2,3,4] \leq 1$.}
  \label{fig:jugdiagram}
\end{figure}

More specifically, the ``essential'' set of rank conditions are those that
are not implied by single other rank conditions. In matroid terminology,
a rank condition not implied by one from a larger subset is a \defn{flat},
and a rank condition not implied by one from a smaller subset is \defn{cyclic}
(a union of ``circuits''); don't confuse this with our ``cyclic intervals''!
The cyclic flats are of additional interest 
because they form a lattice \cite{BdM}. (There is a slight confusion that
the whole $[n]$ may be a cyclic flat, but will not be an ``essential''
interval.)

However, it is possible for a cyclic flat's condition 
to be implied by a {\em combination}
of other conditions, in two ways. If $F = F_1 \coprod F_2$ and 
$\rank(F) = \rank(F_1) + \rank(F_2)$, then $F$ is called \defn{not connected}.
(Example: let $k=3$, $n=6$,
$$ J(i) = 
   \begin{cases} i+5 &\text{$i$ odd} \\ i+1 & \text{$i$ even.} \end{cases} $$
Then $[1,4] = [1,2] \coprod [3,4]$ is not connected. 
Again, don't confuse this with ``contiguous'', which the numbers $[1,4]$ 
certainly are!) The same can happen in the dual matroid, in which case $F$ is
\defn{not nnected}. Following \cite{FS}, call the rank conditions associated 
to the connected and nnected flats of a matroid the \defn{crucial conditions}.
\junk{
Check this:

(In the matroids implicitly considered in \cite{Fulton92},
the essential conditions are already crucial.)
}

The positroid varieties form a stratification of $\Grkn$, 
and a ranked poset where the rank is given by the dimension of the variety.
More specifically, the assignment $\Pi_J \mapsto J$ gives an 
anti-isomorphism of this poset to an order ideal in affine $S_n$ Bruhat order
\cite[Theorem 3.16]{KLS}.
In this poset, we have a covering relation $J \gtrdot J'$ iff $J,J'$
agree away from rows $i$ and $j$ (mod $n$), if the dot $(i,J'(i))$ is
Northeast of the dot $(j,J'(j))$ (which is well-defined, even with periodicity),
and there are no dots in $J$ or $J'$ in the interior of the rectangle
with those as Northeast and Southwest corners. Also,
the difference $r(J) - r(J')$ of the rank functions
is an upper triangular periodic matrix of $0$s and $1$s,
with $1$s in the rectangle $[i+1,j] \times [J(i),J(j)-1]$ (and its
periodic copies). 

\subsection{Several classes of positroid varieties, including interval}

Because of its periodicity, the affine permutation matrix of a bounded
juggling pattern $J$ is determined by what it does in rows $i=1,\ldots,n$,
whose intersection with the strip $i \leq j \leq i+n$ is a parallelogram.
Cut it into a left half ($j\leq n$) and right half ($j>n$).
We can pick out several important classes of positroid varieties, 
with decreasing specificity:
\begin{itemize}
\item \emph{Opposite Schubert varieties.} If the dots run NW/SE in the
  entire parallelogram.
\item \emph{Richardson varieties.} If the dots run NW/SE in each 
  of the two triangular halves.
\item \emph{Interval positroid varieties.} If the dots run NW/SE in
  the right half.
\end{itemize}

We now prove these characterizations, starting with the third.

\begin{Proposition}\label{prop:intdefsagree}
  If $\Pi_f$ is the interval positroid variety associated with the partial 
  permutation $f$, then there exists a unique bounded juggling pattern $J(f)$
  such that $\Pi_f = \Pi_{J(f)}$. Moreover, $J(f)$'s permutation matrix is 
  characterized by having $f$ on the triangle $1 \leq i \leq j \leq n$,
  and has $k$ dots arranged NW/SE on the triangle $i \leq n < j \leq i+n$.
\end{Proposition}

\begin{proof}
  Call the two triangles ``$f$'s triangle'' and ``the second triangle''.
  Let $C_f,R_f \in {[n]\choose n-k}$ be the set of nonzero columns
  and rows of $f$. Then $C_f \geq R_f$ in Bruhat order, 
  by the condition that $f$ is upper triangular.

  We need to construct $J := J(f)$. Copy $f$ into $f$'s triangle, and 
  cross out the complete row and column of each of $f$'s $n-k$ dots.
  The square $i \in [1,n], j\in [n+1,2n]$ will have $k$ remaining
  rows and columns, $r_1 < r_2 < \ldots < r_k \subseteq [1,n], 
  c_1 < \ldots < c_k \subseteq [n+1,2n]$. 
  We'll place dots NW/SE so in matrix entries $(r_i,c_i)$, $i=1\ldots k$, 
  and copy them periodically to $(r_i + Mn, c_i + Mn), M\in \integers$.

  Now we claim that for each $i$, $c_i - n \leq r_i$, 
  i.e. that the $J$ so constructed is a {\em bounded} juggling pattern.
  To see this, let $R'_f = [n] \setminus R_f$ be the rows of the new dots,
  and $C'_f = [n] \setminus C_f$ be the columns minus $n$.
  Since $C_f \geq R_f$, we learn $C'_f \leq R'_f$, i.e. $c_i - n \leq r_i$
  for each $i$. 

  To show $\Pi_f = \Pi_J$, we will show there are no 
  crucial rank conditions in the second triangle.
  Let $(i,j+n)$ be an ``essential'' box there, corresponding to 
  the cyclic interval $[i,j] := [i,n] \coprod [1,j]$.
  (Note that $j<i-1$, since the box $(i-1,j+n)$ North of $(i,j+n)$ must
  both be crossed out by some dot strictly to its East $(i-1,j+n+m)$,
  so $j+1 \leq j+m \leq i-1$.) Then since there are no dots Southwest
  of $(i,j+n)$ in the second triangle (by the NW/SE condition),
  its rank condition is
  \begin{eqnarray*}
    r(M)_{ij}
    &=& |[i,n]| + |[1,j]| 
    - \#\{\text{dots in $J$ that are weakly Southwest of $(i,j+n)$}\} \\
    &=& |[i,n]| + |[1,j]| 
    - \#\{\text{weakly SW of $(i,n)$}\} 
    - \#\{\text{weakly SW of $(1+n,j+n)$}\} \\
    &=& r(M)_{in} + r(M)_{1j}
  \end{eqnarray*}
so the cyclic flat $[i,n] \coprod [1,j]$ is not a connected flat.
\end{proof}

If $J$'s crucial conditions are all {\em intervals}, not just cyclic intervals,
then $\Pi_J$ is obviously an interval positroid variety as defined in
\S \ref{ssec:introIP}. The example given in the last section, whose 
crucial intervals are $[1,2],[3,4],[5,6]$ show that the ``essential''
cyclic intervals may be properly cyclic ($[5,2]$ in that example).

This construction suggests we define \defn{the diagram of a partial
  permutation $f$} by crossing out strictly West and South from each dot,
and also crossing out entirely any row or column with no dot (as secretly,
that dot is hiding in the second triangle of $J(f)$). Then as before, $f$'s
``essential'' boxes are the Northeast corners of the diagram.

The following is essentially well-known; we only include it to fix notation.

\begin{Lemma}\label{lem:oppSchubert}
  If $f$'s dots are in the first $n-k$ rows, running NW/SE, then
  $\Pi_f = X^\lambda$ where the partition $\lambda$ is constructed
  from $f$'s columns, {\em read backwards}, as follows: 
  Start at the point $(0,-k)$
  in the fourth quadrant of the Cartesian plane, and move right for
  each nonzero column, and up for each zero column.
\end{Lemma}

\begin{proof}
  Crossing out South and East from $f$'s dots, and crossing out the $k$
  empty rows beneath, already only leaves a partition in the Northeast corner.
  Crossing out empty columns cuts that into a bunch of partitions,
  each of which reach up to the top row. Hence the essential conditions
  are all on intervals $[1,j]$, and so define an opposite Schubert variety,
  easily checked to be this one.
\end{proof}

\begin{Lemma}\label{lem:Richardson}
  Let $f$ be the left half of $J$'s parallelogram.
  The unique smallest Richardson variety containing $\Pi_J$ 
  is $X_\mu \cap X^\nu$, where $\nu$ is constructed from $f$'s nonzero columns
  as in lemma \ref{lem:oppSchubert}, and $\mu$ by using $f$'s nonzero rows.
  The containment $X_\mu \cap X^\nu \supseteq \Pi_J$ is an equality
  iff the dots run NW/SE in each of the
  left and right halves of $J$'s parallelogram.
\end{Lemma}

\begin{proof}
  First we check straightforwardly that the smallest Schubert and
  opposite Schubert varieties containing $\Pi_J$ are $X_\mu$ and $X^\nu$, 
  by checking the rank conditions on the intervals $\{[i,n]\}$ and $\{[1,j]\}$.

  Since Schubert and opposite Schubert varieties are positroid varieties,
  so are Richardson varieties. So the containment
  $X_\mu \cap X^\nu \supseteq \Pi_J$ is an equality exactly if $\Pi_J$
  is the largest positroid variety with these given
  rank conditions on the intervals $\{[i,n]\}$ and $\{[1,j]\}$.

  If $\Pi_J$ has a NE/SW pair of dots in either the left or right half,
  with the NE dot minimally NE of the SW dot, we can switch them for
  a NW/SE pair by doing a covering relation in affine Bruhat order.
  This terminates when we can't get bigger inside $X_\mu \cap X^\nu$,
  and also when there are no such pairs, as was to be shown.
\end{proof}

Given $X \subseteq \Grkn$, let $\wt X \subseteq M_{k\times n}$ be the
closure of $\{M \in M_{k\times n} : \rowspan(M) \in X\}$,
and call it the \defn{Stiefel cone} over $X$.
(We invented this terminology to generalize the ``affine cone'' 
$k=1$ case, and the Stiefel manifold.) When $k=1$ this is the
usual affine cone over a projective variety. Because of the closure
operation, the Stiefel cone may be more singular than $X$ itself.
Of course our interest is in the case $X=\Pi_J$, where $\wt X = \wt\Pi_J$.

\begin{Proposition}
  Let $f$ be an $n\times n$ upper triangular
  partial permutation matrix of rank $n-k$.
  Construct $f'$ of size $(n+k)\times (n+k)$ of rank $n$,
  by putting $f$ in the upper left corner, $k$ zero rows on the bottom,
  and $k$ dots arranged NW/SE in the remaining $n\times k$ rectangle in the NE.
  Then $\wt\Pi_{J(f)}$ is isomorphic to an open set on $\Pi_{J(f')}$. Hence each
  $\wt\Pi_{J(f)}$ is normal and Cohen-Macaulay, with rational singularities, 
  and intersections of unions of these Stiefel cones are reduced.
\end{Proposition}

\begin{proof}
  The correspondence is $M \mapsto \rowspan [M\ {\rm Id}_k]$, landing
  inside the big cell in which the last Pl\"ucker coordinate is nonzero.
\end{proof}

These good properties do {\em not} hold for the Stiefel cones $\wt\Pi_J$ 
of general positroid varieties. In particular, if the Stiefel cones
over the four positroid divisors in $\Gr{2}{4}$ are $D_1,D_2,D_3,D_4$,
then the scheme $D_1 \cap D_2 \cap (D_3 \cup D_4)$ contains the 
rank $\leq 1$ matrices as a component of multiplicity $2$
(so, nonreduced). From this point
of view, the Stiefel cones of positroid varieties behave as badly as one would 
expect them to, and the Stiefel cones of {\em interval} positroid varieties
are only better behaved because they are open sets on positroid varieties.

This next proposition computes the $T$-fixed points on an interval
positroid variety, in terms of matchings. (It will not be used later.)

\begin{Proposition}\label{prop:matchings}
  Let $f$ be an upper triangular $n\times n$ partial permutation matrix, 
  and $S \subseteq [n]$ a subset. 
  Let $dots(f) = \{(i,f(i)) : f(i)$ defined$\}$.
  Then $\Pi_f$ contains the coordinate space $\AA^{S^c}$ that uses the
  coordinates not in $S$ if and only if there is a matching 
  $m: dots(f) \to S$, where $a \leq m(a,b) \leq b$ 
  for each $(a,b) \in dots(f)$. 

  In words, each dot $(a,b)$ gets matched with a diagonal entry $(m,m)$
  to its Southwest, with $S$ the unmatched part of the diagonal.
\end{Proposition}

\begin{proof}
  Let $n-k$ be the rank of $f$, so by definition, $\Pi_f \subseteq \Grkn$,
  and any coordinate space in it must be $k$-dimensional.
  Thus already $S$ must have size $n-k$. 
  
  By the definition of $\Pi_f$, 
  it contains $\AA^{S^c}$ iff for each interval $[i,j]$, 
  $$ \big|[i,j]\setminus S\big| \ \ \leq\ \ \big|[i,j]\big| - 
  \#\left\{(a,b) \in dots(f) : i\leq a\leq b\leq j\right\}$$
  or equivalently
  $$ \big|[i,j]\cap S\big| \ \ \geq\ \ 
  \#\left\{(a,b) \in dots(f) : i\leq a\leq b\leq j\right\}
  \quad =: \#dots_{\swarrow (i,j)}(f). $$
  If a matching $m$ exists, it gives an injection of the set on the right
  to the set on the left. That proves one direction.

  We will refer to each of these as an ``$[i,j]$ inequality''.
  Assume that each $[i,j]$ inequality holds;
  we need to construct a matching.

  If $f(h)=h$ (i.e. we have a dot on the diagonal), 
  then the $[h,h]$ inequality shows $h\in S$,
  and the matching must include $(h,h) \mapsto h$.
  If we remove the dot from $f$ and $h$ from $S$, producing the new matching
  problem $f',S'$, then any interval $[i,j] \ni m$
  will have both sides of its inequality decrease by $1$,
  and any interval $[i,j] \not\ni m$ will stay exactly the same.
  In particular, the new problem $f',S'$ satisfies the required inequalities,
  so has a solution by induction on $\rank(f)$. With this we reduce to the 
  case that $f$ has no dots on the diagonal. 
  In particular, being strictly upper triangular, it must have some columns 
  without dots. If $f$ is the zero matrix, we are done, so assume otherwise.

  Let $j+1$ be the leftmost column with a dot, say at $(i,j+1)$.  
  We will {\em try} (and possibly fail -- this remains to be seen) 
  to move that dot West to $(i,j)$, producing $f'$. 
  This increases the right-hand side of each $[h,j]$ inequality. If
  they all still hold, then we can use a matching $m'$ for $f'$ to
  build a matching $m$ for $f$, by composing with the correspondence
  between the dots of $f$ and $f'$.

  If some $[h,j]$ inequality does not hold for $f'$, obstructing this move,
  it is because
  $$ \big|[h,j]\cap S\big| \ \ =\ \ 
  \#dots_{\swarrow (h,j)}(f). $$
  Compare with the $[h,j+1]$ inequality; the right side is $1$ dot larger,
  so the left side must have $1$ more element of $S$, i.e. $j+1 \in S$.
  So instead of moving the $(i,j+1)$ West, we will {\em try} to match it up 
  with $j+1 \in S$. (This time, we will be successful.)
  Let $f',S'$ be $f,S$ with $(i,j+1)$ and $j+1$ removed.

  This decreases the left side of various $[a,b]$ inequalities,
  and the right side of others. The only bad possibility is that we
  decrease the left side, but not the right, for some $[a,b]$ inequality
  that held with equality. The left side decreases if $[a,b] \ni j+1$. 
  The right side stays the same if $[a,b] \not\supseteq [i,j+1]$.
  Hence $i < a \leq j+1 \leq b$.

  Let $c$ be the number of dots in the rectangle 
  $[i,a-1] \times [j+1,b]$; we know $c>0$ because of dot $(i,j)$.
  \begin{eqnarray*}
    \left| S\cap [a,j] \right| 
    &=& \left| S\cap [i,j] \right| + \left| S\cap [a,b] \right| 
    - \left| S\cap [i,b] \right| \\
    &=&   \#dots_{\swarrow (i,j)}(f) +   \#dots_{\swarrow (a,b)}(f)
    - \left| S\cap [i,b] \right| \\
    &\leq& \#dots_{\swarrow (i,j)}(f) +   \#dots_{\swarrow (a,b)}(f)
    - \#dots_{\swarrow (i,b)}(f) \\
    &=& \#dots_{\swarrow (a,j)}(f) - c \\
    &<& \#dots_{\swarrow (a,j)}(f)    
  \end{eqnarray*}
  but this contradicts the $[a,j]$ inequality. So there is no obstruction
  to starting the matching with $(i,j+1) \mapsto j+1$.
\end{proof}

If $X \subseteq \Grkn$ is $T$-invariant and irreducible, its \defn{matroid}
is the collection $B(X) := \{S \subseteq [n] : \AA^S \in X \}$.
A \defn{positroid} is one arising
from $X = \overline{T\cdot V}$ where $V$ has all nonnegative real 
Pl\"ucker coordinates \cite{Postnikov}. 
The matroids of the $\{\Pi_f\}$ are positroids, whose connected flats
are intervals (not just cyclic intervals), 
hence the term ``interval positroid variety''.

Under Grassmannian duality $\Grkn \iso \Gr{n-k}{n}$, one can check that
the positroid variety $\Pi_f$ is corresponded with the positroid variety 
$\Pi_{(i \mapsto i+n) \circ f^{-1}}$. This does {\em not} preserve the
subclass of interval positroid varieties (other than the
Richardson varieties). The additional power available
from dualizing is exploited in \cite{KL}.

One interpretation of proposition \ref{prop:matchings} is that
interval positroids are ``dual transversal'' matroids. Consider the dots
in $f$ as a set of $n-k$ choosy brides $b$, each of whom will only marry a
groom within a certain height range $[i_b,j_b]$.
Then each $S$ is an acceptable set of grooms.  
Hall's Marriage Theorem (from which this terminology is derived) 
says that if some groomset $S \subset [n], |S|=\rank(f)$ is unmarriable,
it is because there is a set $B$ of brides with 
$\left| S \cap \union_{(i,j) \in B} [i,j] \right| < |B|$.
Proposition \ref{prop:matchings} goes further in two ways:
it says that if $S$ is unmarriable, then (1) there is an interval $[a,b]$
where $B$ is the brides with ranges in that interval, and 
the grooms in that interval aren't numerous enough, (2) even if one includes 
the grooms none of those brides wants.

\subsection{A Monk formula for positroid varieties}

For $S \subseteq [n]$, $r\in \naturals$, let
$$ X_{S\leq r} := \{\rowspan(M) :   M \in M_{k\times n}, 
\rank\ M = k, \rank\ (\text{columns $S$ of }M) \leq r \}. $$
If $S$ is a cyclic interval, call this a \defn{basic positroid variety}. 
Clearly every positroid variety is an intersection of basic ones, 
and one can show that no basic positroid variety is an intersection of
other positroid varieties.

\begin{Theorem}\label{thm:monk}
  Let $J$ be a bounded juggling pattern, and $\big(i,j = J(i)\big)$
  one of its dots.  Let $C = (j')$ be the columns of those dots
  minimally Northwest of $(i,J(i))$, ordered Northeast/Southwest, and
  $r = |[i,j-1]| - \#\{$dots in $J$ that are weakly Southwest of $(i,j-1)\}$.
  Then
  $$ \Pi_J \cap X_{[i,j-1] < r} 
  = \Union_{j' \in C} \Pi_{J'\circ (j\leftrightarrow j')}. $$
\end{Theorem}

\begin{proof}
  An intersection of positroid varieties is a reduced union of
  positroid varieties \cite[corollary 4.4]{KLS2}, so we just 
  need to determine which such occur in $\Pi_J \cap X_{[i,j-1] < r}$.
  We alert the reader that this is perhaps the subtlest 
  combinatorial argument in the paper.

  By the definition of $C$, 
  each $J \circ (j\leftrightarrow j') \gtrdot J$ 
  is a covering relation in affine Bruhat order, 
  so $\Pi_{J'\circ (j\leftrightarrow j')} \subset \Pi_J$.
  Also, the dot in column $j'$ of $J$ moves down to row $i$, 
  providing another dot weakly Southwest of $(i,j-1)$, 
  hence $\Pi_{J'\circ (j\leftrightarrow j')} \subseteq X_{[i,j-1] < r}$.
  Together these prove the $\supseteq$ containment.

  For the reverse containment, we need to show that for any
  $\Pi_{J'} \subseteq \Pi_J \cap X_{[i,j-1] < r}$, there exists a $j'\in C$
  such that $\Pi_{J'} \subseteq \Pi_{J \circ (j\leftrightarrow j')}$.
  In rank matrix terms, we need $r(J') \leq r(J)$, with 
  strict inequality on some rectangle $[J^{-1}(j')+1,i] \times [j',j-1]$,
  not just at $(i,j-1)$.

  Consider a saturated chain 
  $J = J_0 \lessdot J_1 \lessdot \ldots \lessdot J_k = J'$
  in strong affine Bruhat order, 
  so in particular we have entrywise inequalities 
  $r(J_0) \geq r(J_1) \geq \ldots \geq r(J_k)$,
  and more specifically each $r(J_m) - r(J_{m+1})$ is a matrix of $0$s and $1$s
  with the $1$s in a rectangle. Then there exists a smallest $m$
  such that $r(J_0)_{i,j-1} > r(J_m)_{i,j-1}$, 
  and that $\Pi_{J_m}$ is therefore contained in $X_{[i,j-1] < r}$. 
  With this we can reduce to the case $k=m$.

  To describe our goal another way, each covering relation moves the
  dots at the NW and SE corners of an otherwise empty rectangle to the
  NE and SW corners. We want to show that the union of these rectangles
  from the covering relations in the chain $(J_0,\ldots,J_k)$ contains
  one of the maximal rectangles in the \defn{staircase of $J$ (above $(i,j)$)},
  the set of boxes weakly Southeast of some box of $C$ and weakly Northwest 
  of $(i,j)$. We will prove this by induction on $k$.

  The $k=1$ case is easy -- the only rectangle must include $(i,j)$, 
  so must have it as the Southeast corner, and hence the Northwest corner
  column must be in $C$.

  Consider the corresponding staircases for $J_0,\ldots,J_k$ with
  Northwest corner sets $C = C_0,\ldots,C_{k-1}$. 
  If the covering relation $J_0 \lessdot J_1$ gives an increase 
  in the staircase, or leaves it the same, then we can use induction.
  Otherwise, one checks that one of the dots $d$ in $C$ must move South or East
  inside the staircase to $d'$, as pictured (these being Southern moves):

  \centerline{\epsfig{file=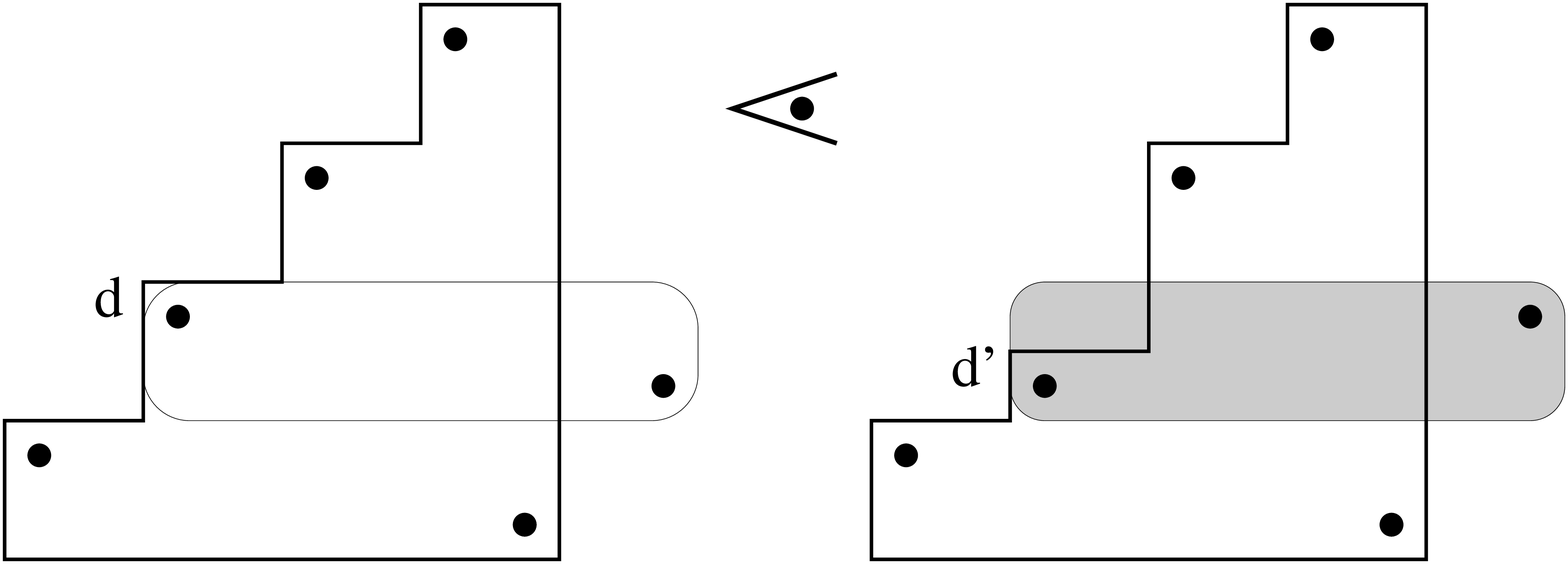, height=1.2in} \qquad\qquad
  \epsfig{file=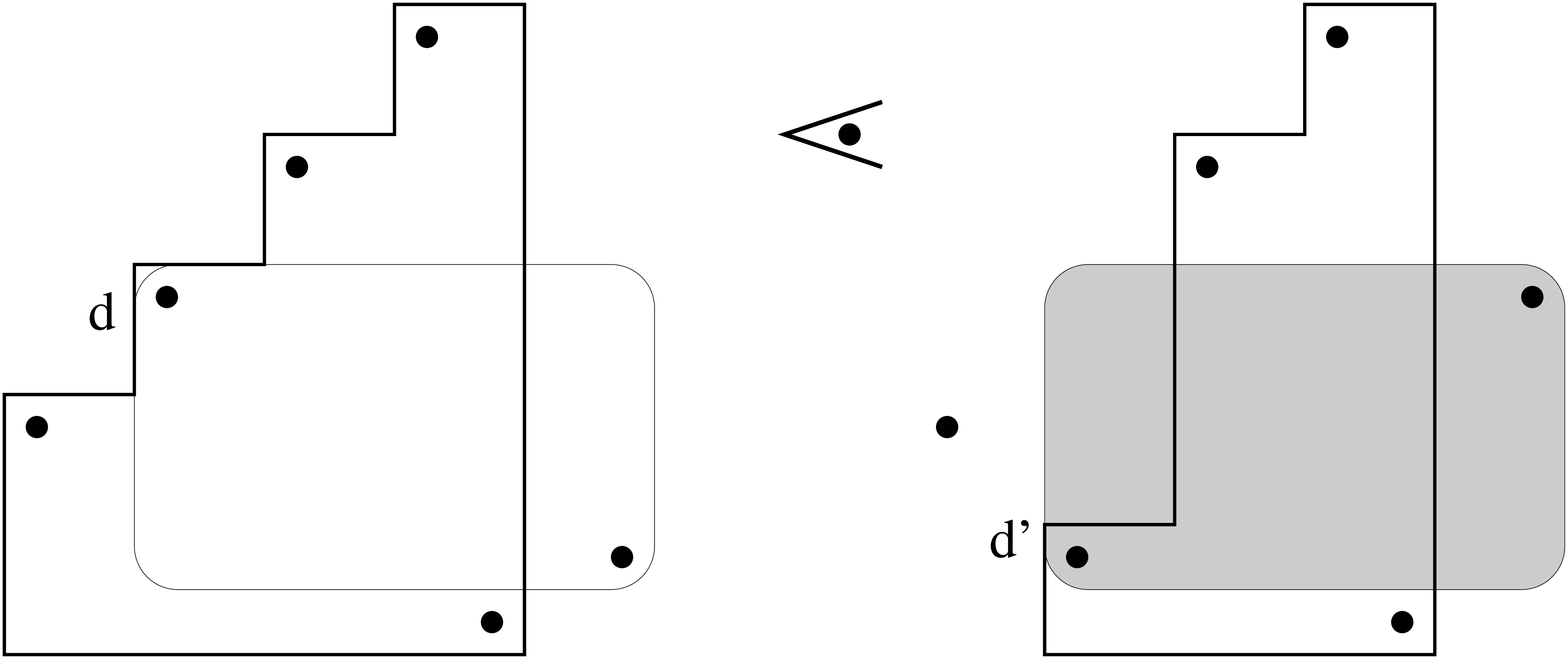, height=1.2in}}

  By induction, the remaining covering relations from $J_1$ to $J_k$
  give rectangles that cover a rectangle $R$ connecting $(i,j)$ to one
  of the NW corners of $J_1$'s staircase. If that corner is not $d'$,
  then it is one of the corners of $J_0$ and we're done. If that corner is $d'$,
  then $R$ union the rectangle acquired during the $d\searrow d'$ move
  covers the rectangle connecting $(i,j)$ to $d$.
\end{proof}

\section{Combinatorial and geometric shifting}\label{sec:shifting}

\newcommand\coord{{coord}}

The classic combinatorial shift operations defined in \cite{EKR}
concern the sets $[n] := \{1,2,\ldots,n\}$ and 
${[n] \choose k} := \{S \subseteq [n] : |S| = k\}$.
Before getting into them, we establish a basic correspondence between
collections of subsets (the combinatorial side) and certain subschemes of 
the Grassmannian (the geometrical side).

\subsection{Between collections and subschemes}

\newcommand\calC{{\mathcal C}}

The connection to geometry begins 
with the correspondence
$$ \coord : {[n]\choose k}\ \widetilde\to\ \Grkn^T, \qquad 
S \mapsto
\coord(S) := 
 \{ \vec v \in \AA^n \text{ that use only coordinates from }S\}
$$
between $k$-subsets and the $T$-fixed points, the coordinate subspaces.

If $X \subseteq \Grkn$ is a closed $T$-invariant subscheme,
not just a point, we can nonetheless look at its fixed points $X^T$,
and write
$$ \coord^{-1}(X) := \coord^{-1}(X^T) = 
\left\{S \in {[n]\choose k}\ :\ \coord(S) \in X^T\right\}
\quad \subseteq {[n]\choose k}. $$
To forestall confusion when talking about sets of sets, we will call
any $S \subseteq [n]$ a \defn{subset} and any  
$\calC \subseteq {[n]\choose k}$ a \defn{collection}.
Extend $\coord$ beyond subsets to collections, as follows:
$$ \coord(\calC) := \bigcap_{S \notin \calC}\ \{V \in \Grkn\ :\ p_S(V) = 0\} $$
where $p_S$ is the Pl\"ucker coordinate. 
This is the ``bracket ring'' construction of \cite{White}, 
in which $\calC$ is somewhat needlessly assumed to be a matroid,
presumably because $\coord(\calC)$ is reducible otherwise.%
\footnote{ It is well-known that if $X$ is irreducible and $T$-invariant, 
  then $\calC := \coord^{-1}(X^T)$ is the bases of a {\defn matroid}, 
  meaning that for each $\pi\in S_n$, the collection $\pi\cdot \calC$
  has a unique Bruhat minimum.
  Proof: Let $\check\rho$ be a regular dominant coweight, 
  so its Bia\l ynicki-Birula decomposition of $\Grkn$ is the 
  Bruhat decomposition. If $X$ is irreducible, then for each 
  $\pi \in S_n$, $\pi\cdot X$ will have a unique open Bia\l ynicki-Birula 
  stratum, whose center is this unique Bruhat minimum.
  See e.g. \cite{CoxeterMatroids}.}

To study these operations,
we first need a basic result about Pl\"ucker coordinates:

\begin{Lemma}\label{lem:Plucker}
  Let $X \subseteq \Grkn$ be $T$-invariant and reduced.
  If $S \in {[n]\choose k}$, and $\coord(S) \notin X$, 
  then the Pl\"ucker coordinate $p_S$ vanishes on $X$.
\end{Lemma}

\begin{proof}
  Consider the one-parameter subgroup 
  $$ d:\Gm \to T,\quad t \mapsto \diag(t^{d_1},\ldots,t^{d_n}) 
  \qquad \text{where} 
  \quad d_k = \begin{cases} 
    0 &\text{if $k\in S$} \\ 
    1 &\text{otherwise} \end{cases} $$
  The sink of $d$'s Bia\l ynicki-Birula decomposition of $\Grkn$ is 
  the point $\coord(S)$, and its basin of attraction is the big cell 
  $p_S \neq 0$. If $X$ meets this cell, then since $X$ is $d$-invariant
  (being $T$-invariant) and closed, $X \ni \coord(S)$, contradiction.
  Hence $X$ is set-theoretically contained in the divisor $\{p_S = 0\}$,
  and since it was assumed reduced $X$ is contained in that divisor
  scheme-theoretically as well.  
\end{proof}

\begin{Proposition}\label{prop:connection}
  For $C \subseteq {[n]\choose k}$, 
  $$ \coord^{-1}(\coord(C)^T) = C. $$
  For $X \subseteq \Grkn$, 
  closed, reduced and $T$-invariant but possibly reducible,
  $$ X  \subseteq  \coord\left( \coord^{-1}(X) \right). $$
\end{Proposition}

\begin{proof}
  The first is tautological:
  \begin{eqnarray*}
    \coord(C)^T 
    &=& \bigcap_{S\notin C}\ \{V \in \Grkn^T : p_S(V) = 0\}
    = \bigcap_{S\notin C}\ \{V \in \Grkn^T : V \neq \coord(S) \}    \\
    &=& \{V \in \Grkn^T : \coord^{-1}(V) \notin \{S : S\notin C\} \} \\
    &=& \{V \in \Grkn^T : \coord^{-1}(V) \in C \}
  \end{eqnarray*}
  The second is essentially a restatement of lemma \ref{lem:Plucker}.
\end{proof}

In particular, the assignment $\calC \mapsto \coord(\calC)$ 
corresponds $PowerSet({[n]\choose k})$ 
with a certain collection of $T$-invariant subschemes of $\Grkn$,
with inverse correspondence $X \mapsto X^T$. A subscheme is in
the collection exactly if it is defined by the vanishing
of Pl\"ucker coordinates.

It is a classical theorem of Hodge and Pedoe that Schubert varieties
are subschemes of this type. The same is true more generally of
positroid varieties \cite[corollary 5.12]{KLS}, and will also be true for the
{\em reducible} schemes that we will produce through geometric shifting.

\subsection{Combinatorial shifting}\label{ssec:combshift}

Let $m \in [n]$, $S \subseteq [n]$, $\calC \subseteq PowerSet([n])$
be an element, subset, and collection respectively. 
At each of these three levels, the shifting mantra is

\begin{quote}
  ``turn $i$ into $j$, unless something's in the way''.
\end{quote}

\noindent (At the single-element level, nothing can be in the way.)
\begin{eqnarray*}
  \sh_{i\to j} m &:=& 
  \begin{cases} m &\text{if $m\neq i$} \\ j    &\text{if $m=i$}\end{cases} \\
  \sh_{i\to j} S &:=& 
  \begin{cases} \sh_{i\to j} m &\text{if $\sh_{i\to j} m \notin S$} \\
    m &\text{if $\sh_{i\to j} m \in S$} \end{cases} \quad:\quad m \in S \bigg\}\\
  \sh_{i\to j} \calC &:=& 
  \begin{cases} \sh_{i\to j} S &\text{if $\sh_{i\to j} S \notin \calC$} \\
    S &\text{if $\sh_{i\to j} S \in \calC$} \end{cases} 
        \quad:\quad S \in \calC \bigg\} 
\end{eqnarray*}
In particular, if $S = \{m\}$ is a singleton then $\shij S = \{\shij m\}$,
and likewise if $\calC = \{S\}$ is a singleton then $\shij \calC = \{\shij S\}$,
but in general the shift of a set or collection is {\em not} just the
shift of its elements. We leave the reader to check the following:

\begin{Lemma}\label{lem:shiftlevels}
  $$|\sh_{i\to j} S| = |S|, \qquad
   \sh_{i\to j} S \supseteq \{\sh_{i\to j} m : m\in S\}, \qquad
   \sh_{i\to j} S \setminus \{\sh_{i\to j} m : m\in S\} = 
  \begin{cases}
    i & \text{if $i,j \in S$} \\
    \emptyset & \text{otherwise}
  \end{cases} $$
  $$|\sh_{i\to j} \calC| = |\calC|, \quad
  \sh_{i\to j} \calC \supseteq \{\sh_{i\to j} S : S\in \calC\}, \quad
  \sh_{i\to j} \calC \setminus \{\sh_{i\to j} S : S\in \calC\} =
  \{ S \in \calC : S \neq \sh_{i\to j} S \in \calC \}. $$
\end{Lemma}

The Erd\H os-Ko-Rado theorem does not really study shifting itself as a process,
so much as collections $C$ that are invariant under all forward shifts,
and there is an industry of combinatorial results concerning
various objects (collections, matroids, simplicial complexes) 
that are ``shifted'' (see e.g. \cite{Frankl,Kalai}).
There does not seem to be as much study of the incremental shifting
we make use of here.

\subsection{Geometric shifting}\label{ssec:geomshift}

Hereafter $X$ is a closed subscheme of $\Grkn$, and almost always $T$-invariant.
\junk{
Under the loose correspondence 
$X \mapsto X^T \subseteq \Grkn^T \iso {[n]\choose k}$,
it will be the geometric analogue of a collection of $k$-element subsets.

Denote the latter bijection 
$\coord : {[n]\choose k}\ \widetilde\to\ \Grkn^T$, taking
a subset to the corresponding coordinate subspace.
}
Before defining the shift, first define
$$ \wt\swij X 
:= \overline{\{ (t,\exp(te_{ij})\cdot x) : t\in \AA^1, x\in X\}}
\quad \subseteq \PP^1 \times \Grkn $$
where the closure adds the fiber at $t=\infty$,
and define the \defn{(geometric) shift} $\sh_{i\to j} X$ to be this 
scheme-theoretic fiber over $\infty$ of the (automatically flat)
projection to $\PP^1$.
The shift need not be reduced; if $X$ is the two points $\Gr{1}{2}^T$,
then one falls into the other during the shift, and
$\sh_{1\to 2} X$ is a double point.
The \defn{(geometric) sweep} $\swij X$ is defined as the image of
the projection of $\wt\swij X$ to $\Grkn$. 
The same example $\Gr{1}{2}^T$ shows that this 
projection need not be birational to its image.

\junk{
  The projection $\wt\Psi_{i\to j} X \to \Grkn$, with image $\Psi_{i\to j} X$,
  need not be generically $1$:$1$; let the \defn{degree of the sweep}
  be the degree of this map, or zero if $X$ is shift-invariant.
}

Having defined the geometric analogue of the shift of a collection,
we can (in analogy to the paragraph before lemma \ref{lem:shiftlevels})
deduce the analogues of the shifts of elements and subsets:

\begin{Lemma}\label{lem:lowergeomshifts}
  Let $X = \{V\} \subseteq \Grkn$ (the analogue of $\calC = \{S\}$). Then 
  $$ \shij \{V\} = 
  \begin{cases}
    \{V\} &\text{if } V \leq \coord({[n]\setminus i}) 
    \text{ or } V \geq \coord(\{j\}) \\
    \left\{\left(V \cap \coord({[n]\setminus i})\right) 
      \oplus \coord({\{j\}})\right)
    &\text{otherwise.}
  \end{cases}
  $$
  If in addition $V$ is one-dimensional (the analogue of $S = \{m\}$), then
  $$ \shij \{V\} = 
  \begin{cases}
    V &\text{if } V \leq \coord({[n]\setminus i}) \\
    \coord({\{j\}}) &\text{otherwise.}
  \end{cases}
  $$
\end{Lemma}

\begin{proof}
  We prove the first, from which the second is an evident special case.
  Pick a basis for $V \cap \coord({[n]\setminus i}))$, 
  and if $V \not\leq \coord({[n]\setminus i})$, extend to a basis of $V$.
  If we make these basis vectors the row vectors of a $k\times n$ matrix,
  then the $i$th column is $0$ except possibly in the last row.

  The action of $\exp(te_{ij})$ adds $t$ times column $i$ to column $j$. 
  If the $i$th column is zero, nothing happens. Otherwise we can add
  $t$ times column $i$ to column $j$, then (without changing the row span)
  scale the last row by $t^{-1}$ (for $t\neq 0$). As $t\to\infty$ the
  last row converges to the vector with $1$ in column $j$, $0$ elsewhere.
\end{proof}


The following proposition, essentially the reason \cite{Vakil} brought
shifting into Schubert calculus, will be the means by which we can inductively 
compute the class of an interval positroid variety.

\begin{Proposition}\label{prop:KTshift}
  Let $X \subseteq \Grkn$ be $T$-invariant and irreducible. Then
  $$ [X] = [\shij X] \qquad\qquad\qquad\qquad\qquad 
  \text{in $H^*(\Grkn)$ or $K(\Grkn)$.} $$
  If the map $\wt\swij X \to \swij X$ is degree $1$
  (as will be checkable using theorem \ref{thm:Tconvex} to come), then 
  $$ [X] = [\shij X] + (y_i-y_j) [\swij X] \qquad\qquad\qquad
  \text{in $H^*_T(\Grkn)$.} $$
  If in addition $\swij X$ has rational singularities, then
  $$ [X] = \exp(y_j-y_i) [\shij X] + (1 - \exp(y_j-y_i)) [\swij X] \qquad
  \text{in $K^*_T(\Grkn)$.} $$
\end{Proposition}

\begin{proof}
  Consider the projection $\swij X \onto \PP^1$
  to the first factor. If we act on $\AA^1 \subset \PP^1$ with weight 
  $y_i-y_j$, then this map is $T$-equivariant. 

  Let $[0], [\infty]$ denote the classes of these points in $\PP^1$
  in the various cohomology theories. 
  Then nonequivariantly we have $[0] = [\infty]$, 
  in $H^*_T(\PP^1)$ we have $[0] = [\infty] + (y_i-y_j) [\PP^1]$ 
  and in $K^*_T(\PP^1)$ 
  we have\footnote{Perhaps the most mnemonic way to think of this is in
    terms of the Atiyah-Bott localization formula in $K$-theory, which gives
    $$ [\PP^1] = \frac{[0]}{1 - \exp(-wt(T_{ 0} \PP^1))}
    + \frac{[\infty]}{1 - \exp(-wt(T_{ \infty} \PP^1))}
    \qquad \in K_T(\PP^1) \tensor frac(K_T(pt)). $$
  }
  $[0] = \exp(y_i-y_j) [\infty] + (1-\exp(y_i-y_j)) [\PP^1]$.

  Now pull whichever equation back to $\wt\swij X$, where
  $[0],[\infty],[\PP^1]$ pull back (in any cohomology theory) 
  to $[X \times 0]$, $[\shij X \times \infty]$, and $[\wt\swij X]$.

  Then push this equation forward 
  to $\Grkn$, where
  $[X \times 0],[\shij X \times \infty]$ push forward to $[X],[\shij X]$.
  If the degree of $\wt\shij X \to \shij X$ is $k$, then 
  the fundamental class $[\wt\shij X]$ pushes forward to $k [\shij X]$
  in $H^*_T(\Grkn)$, and in general to something very complicated 
  in $K^*_T(\Grkn)$.

  However, if the degree is $1$ 
  (so that the induced map $\wt\swij X \to \swij X$ takes 
  the structure sheaf to the structure sheaf)
  and $\shij X$ has rational singularities
  (so that there are no higher direct images in sheaf cohomology),
  then $[\wt\swij X]$ pushes forward to $[\swij X]$, and we are done.
\end{proof}

To compute the shift and sweep we will obtain upper bounds from algebra, 
and lower bounds from geometry, which will sometimes coincide.

\begin{Proposition}\label{prop:shiftcupcap}
  Let $X_1, \ldots, X_m \subseteq \Grkn$ be $T$-invariant. 
  Then $\sh_{i\to j} \Union_k X_k \supseteq \Union_k \sh_{i\to j} X_k$,
  with equality as sets. Also, 
  $\sh_{i\to j} \bigcap_k X_k \subseteq \bigcap_k \sh_{i\to j} X_k$, but may
  be unequal as sets.
\end{Proposition}

\begin{proof}
  It is easy to see that $\wt\swij \Union_k X_k = \Union_k \wt\swij X_k$
  as schemes. Intersecting with $\{\infty\} \times \Grkn$, we get
  \begin{eqnarray*}
   \shij \Union_k X_k &=& (\{\infty\} \times \Grkn) \cap \Union_k \swij X_k \\
   &\supseteq& \Union_k (\{\infty\} \times \Grkn) \cap \wt\swij X_k
   \qquad \text{with equality as sets} \\
   &=& \Union_k \shij X_k.
  \end{eqnarray*}

  The latter inequality follows from the fact that an intersection of
  closures (the ones defining $\wt\swij \Union_k X_k$ 
  and each $\wt\swij X_k$) is contained in the closure of the intersection.
\end{proof}

The example $X_1 = \{0\}, X_2 = \{\infty\}$ in $\PP^1$
shows that both containments in proposition \ref{prop:shiftcupcap} 
can be strict (the first scheme-theoretically, 
the second even set-theoretically).

\begin{Proposition}\label{prop:sweep}
  Let $Y \subseteq \Grkn$ be $\sh_{i\to j}$-invariant and irreducible.
  Let $X \subset Y$ be a divisor and {\em not} $\sh_{i\to j}$-invariant.
  Then $\sw_{i\to j} X = Y$.
\end{Proposition}

\begin{proof}
  By $Y$'s $\sw_{i\to j}$-invariance, $\swij X \subseteq Y$.
  Since $X$ is not $\shij$-invariant, $\dim \swij X = \dim Y$.
  (In particular $\swij X$ is nonempty!)
  By $Y$'s irreducibility, $\swij X = Y$.
\end{proof}

\subsection{Connecting the two shifts}\label{ssec:connectshifts}

\junk{
To connect the combinatorial and geometric shifts, let $n\choose k$
denote the {\em collection} of $k$-element subsets, 
and for $S \in {n\choose k}$ let $coord(S) \in \Grkn$ denote the 
corresponding coordinate subspace. This gives a bijection 
$coord: {n\choose k} \ \wt\to\ \Grkn^T$, where the target is
the $T$-fixed point set.
}

We're now ready to compare the geometric and combinatorial shifts.
In a particularly simple case, we can guarantee equality.

\begin{Lemma}\label{lem:shiftSback}
  If $|S| = k$, then
  $ \shij\ \{V : p_S(V) = 0\} = \{V : p_{\shji S}(V) = 0\}. $

  More generally, for $S \subseteq [n]$ of any size, and
  $$ X_{S\leq r} := \{\rowspan(M)\ :\   M \in M_{k\times n},\
  \rank\ M = k,\ \rank\ (\text{columns $S$ of }M) \leq r \}, $$
  we have $ \shij\ X_{S\leq r} = X_{(\shji S) \leq r}$, i.e. 
  ``rank conditions shift backwards''.
\end{Lemma}

\begin{proof}
  The first is the special case $|S|=k, r=k-1$ of the second. 
  Let 
  \begin{eqnarray*}
    Y_t &:=& \exp(t e_{ij}) \cdot \{M \in M_{k\times n} :
    \rank\ (\text{columns $S$ of }M) \leq r \} \\
    &=& \{M \in M_{k\times n} :
    \rank\ (\text{columns $S$ of } M \exp(-t e_{ij})) \leq r \}. 
  \end{eqnarray*}
  The matrix $M \exp(-t e_{ij}))$ matches $M$, except the $j$th 
  column $\vec m_j$ has been replaced by $\vec m_j - t \vec m_i$.
  The shift is $\lim_{t\to\infty} Y_t$.

  If $j\notin S$, then $Y_t$ puts no contraints on column $j$, 
  so $Y_t = Y_0 = X_{S\leq r}$ for all $t$, and $\shij X_{S\leq r} = X_{S\leq r}$.
  In this case $S = \shji S$, too.

  If $i,j\in S$, then subtracting $t$ times column $j$ from column $i$
  doesn't change the rank of columns $S$,
  so $Y_t = Y_0 = X_{S\leq r}$ for all $t$, and $\shij X_{S\leq r} = X_{S\leq r}$.
  Again, $S = \shji S$.

  The interesting case is $j\in S,i\notin S$. For $t\neq 0$, the rank
  of columns $S$ in $M\in Y_t$ doesn't change if we divide column $j$ by $-t$.
  So the rank condition is now
  $$ \rank \bigg( \{\vec m_i - t^{-1} \vec m_j\} 
  \ \cup\ \{\vec m_k : k \in S\setminus j\} \bigg) \leq r. $$
  In the limit, this becomes $\rank($columns $\shji S)\leq r$.

  Effectively, we have found some equations that hold on $\wt\swij X_{S\leq r}$,
  and intersected them with the $t=\infty$ fiber,
  showing the inclusion $\shij X_{S\leq r} \subseteq X_{\shji S\leq r}$. 
  Since $\shij X_{S\leq r}$ is a flat limit of $X_{S\leq r}$, they must have
  the same Hilbert polynomial (with respect to the Pl\"ucker embedding).
  Meanwhile, $X_{\shji S\leq r} = (i\leftrightarrow j) \cdot X_{S\leq r}$,
  so $X_{\shji S\leq r}$ also has this same Hilbert polynomial.
  Consequently the inclusion of schemes is equality.
\end{proof}

In the most general case, we have an inequality:

\begin{Proposition}\label{prop:shiftCback}
  Let $X \subseteq \Grkn$ be $T$-invariant. Then 
  $$ \coord^{-1}(\shij X) \quad\subseteq\quad \shij \coord^{-1}(X). $$
  If $X$ is reduced, then
  $$ \shij X \quad \subseteq\
  \coord(\shij \coord^{-1}(X)) $$
  which, by proposition \ref{prop:connection}, implies the first containment.
\end{Proposition}

\begin{proof}
  Note that neither side of the first claim changes if we replace $X$ 
  by its reduction. So we can assume $X$ reduced in both claims. 
  We want to show 
  $$ \text{for all $S \notin \shij \coord^{-1}(X)$}, \qquad
  \shij X \subseteq X_{S < k} $$
  as the intersection of those divisors $X_{S<k}$ defines the right-hand side
  of the second claim.

  If $S\notin \coord^{-1}(X)$, 
  then $X \subseteq X_{S<k}$ by lemma \ref{lem:Plucker},
  and $\shij X \subseteq \shij X_{S<k} = X_{\shji S < k}$ 
  by lemma \ref{lem:shiftSback}. Hence $\shji S \notin \coord^{-1}(\shij X)$.

  If $S \supseteq \{i,j\}$ or $S \cap \{i,j\} = \emptyset$, 
  then $S = \shij S = \shji S$ and $S \in \calC \iff S \in \shij \calC$. 
  In particular, we've shown for these $S$ that
  $S \notin \shij \coord^{-1}(X) \implies \shij X \subseteq X_{S < k}$.

  It remains to consider those $S$ that contain $i$ or $j$ but not both,
  which we will do in pairs.
  Let $M$ vary over ${[n]\setminus \{i,j\}\choose k-1}$
  and look at $\shij \coord^{-1}(X) \ \cap\ \{M\cup \{i\}, M\cup \{j\}\}$,
  which may be $\emptyset$, or $\{M\cup \{j\}\}$, 
  or $\{M\cup \{i\}, M\cup \{j\}\}$, but not $\{M \cup \{i\}\}$.

  In the first case, 
  $\coord^{-1}(X) \cap \{M\cup \{i\}, M \cup \{j\}\} = \emptyset$ too.
  Hence $X \subseteq X_{M \cup \{i\} < k} \cap X_{M \cup \{j\} < k} 
  = X_{M < k-1}\, \cup\, X_{M \cup \{i,j\} < k}$,
  and the latter union is visibly shift-invariant,
  so $\shij X \subseteq X_{M \cup \{i\} < k} \cap X_{M \cup \{j\} < k}$ too.

  In the second case
  $S = M \cup \{i\} \notin \shij \coord^{-1}(X) \ni M \cup \{j\}$,
  so $\coord^{-1}(X) \not\supseteq \{M\cup \{i\}, M\cup \{j\}\}$.
  Whichever one is missing, $h=i$ or $j$, gives us a containment 
  $X \subseteq X_{M \cup \{h\}<k}$. 
  Shifting it, we learn $\shij X \subseteq \shij X_{M \cup \{h\}<k} 
  = X_{\shji (M \cup \{h\})<k} = X_{M \cup \{i\}<k} = X_{S < k}$.

  In the third case, our $S \notin \shij \coord^{-1}(X)$ 
  can be neither of $\{M\cup \{i\}, M\cup \{j\}\}$, so there
  is nothing left to prove.
\end{proof}

These containments are strict for $X = \Gr{1}{2}^T \iso \{0,\infty\}$,
where $\sh_{1\to 2} \,\coord^{-1}(X) \neq \coord^{-1}(\sh_{1\to 2} X)$,
so we'll need a condition, ``$T$-convexity'', to rule out such examples.

The only $T$-invariant irreducible curves in $\Grkn$ are of the form
$$ L = \{V \in \Grkn : \coord(M) < V < \coord(M \cup \{i,j\}), \qquad
M \in {[n]\setminus \{i,j\} \choose k-1} $$
connecting the two fixed points
$L^T = \{\coord(M\cup\{i\}), \coord(M\cup\{j\})\}$.
Call a subset $X\subseteq \Grkn$ \defn{$T$-convex} if 
$L^T \subseteq X \implies L \subseteq X$ for each such $L$.

\begin{Theorem}\label{thm:Tconvex}
  Let $X\subseteq \Grkn$ be $T$-invariant.
  \begin{enumerate}
  \item If $X$ is irreducible, then $X$ is $T$-convex.
  \item If $X$ is defined by the vanishing of a set of Pl\"ucker coordinates,
    then $X$ is $T$-convex.
  \item If $X$ is $T$-convex, then
    $$ \coord^{-1}(\sh_{i\to j} X)^T = \left( \sh_{i\to j} \coord^{-1}(X^T) \right).$$
  \item If $X$ is irreducible, and the collection
    $\coord^{-1}(X^T)$ is not $\sh_{i\to j}$-invariant, 
    then the map $\wt\sw_{i\to j} X \to \sw_{i\to j} X$ is a degree $1$ map
    of varieties.
  \end{enumerate}
\end{Theorem}

\begin{proof}
  \begin{enumerate}
  \item Let $L = \{V \in \Grkn : \coord(M) < V < \coord(M \cup \{i,j\})$
    be a line such that $L^T \subset X$.
    Consider the one-parameter subgroup 
    $$ d:\Gm \to T,\quad t \mapsto \diag(t^{d_1},\ldots,t^{d_n}) 
    \qquad \text{where} 
    \quad d_k = \begin{cases} 
      0 &\text{if $k\in M$} \\ 
      1 &\text{if $k=i,j$} \\
      2 &\text{otherwise} \end{cases} $$
    which fixes $L$ pointwise.
    Under $\Grkn$'s Bia\l ynicki-Birula decomposition \cite{BB} using $d$,
    the sink is $L$. (One way to compute the sink is as the minimum
    level set of $d$'s moment map, which takes $\coord(Q) \mapsto \sum_Q d_q$.)

    Now obtain $X$'s B-B decomposition by intersecting with $\Grkn$'s.
    Since $X \cap L$ is $T$-invariant and contains $L^T$, either $X \supseteq L$
    or $X\cap L = L^T$. In the latter case, each of $L^T$'s two points 
    gives a sink in $X$, hence two disjoint open basins,
    contradicting $X$'s irreducibility.
  \item
    The Schubert divisor, defined by the vanishing of the first Pl\"ucker
    coordinate, is irreducible, hence $T$-convex. The other Pl\"ucker divisors
    are permutations of the Schubert divisor, hence $T$-convex.
    The intersection of two $T$-convex sets is again $T$-convex.

    (In fact such a scheme is even ``convex'': for any two points in
    $X$ connected by a line in $\Grkn \subseteq \PP(Alt^k \AA^n)$, 
    the whole line is in $X$.
    Not every irreducible $T$-invariant $X$ is convex; consider the
    subvariety of $\Gr{2}{4}$ defined by $p_{12} p_{34} = p_{14} p_{23}$,
    and the non-$T$-invariant pencil
    $$ \left\{ \rowspan\begin{bmatrix} 1&0&1&a\\ 0&1&a&1 \end{bmatrix} 
      : a\in\AA^1\right\}. $$
    On there, the equation becomes $1(1-a^2) = 1(-1)$, with 
    two solutions $a = \pm 1$.)
  \item We already have the $\subseteq$ containment by proposition 
    \ref{prop:shiftCback}. Also,
    \begin{eqnarray*}
      \sh_{i\to j} X 
      &\supseteq& \sh_{i\to j} (X^T) 
      = \sh_{i\to j} \Union_{\coord(S) \in X^T} \coord(S) 
      \supseteq \Union_{\coord(S) \in X^T} \sh_{i\to j} \coord(S) \\
      &=& \Union_{\coord(S) \in X^T} \coord(\sh_{i\to j} S)
      = \coord\left \{\sh_{i\to j} S : \coord(S) \in X^T \right\}.
    \end{eqnarray*}
    By lemma \ref{lem:shiftlevels},
    this last is contained in $\coord\left( \sh_{i\to j} \coord^{-1}(X^T) \right)$,
    and the set difference is 
    $$ \{\coord(S) 
        : \exists S \neq \shij(S),\ S,\shij(S) \in \coord^{-1}(X^T) \}. $$
    If there is such an $S$, let $M = S \setminus \{i\} = S \cap \shij(S)$, 
    and $L = \{V \in \Grkn : \coord(M) < V < \coord(M \cup \{i,j\})$.
    So far we've determined that $L^T \subseteq X$. Since $X$ is
    assumed $T$-convex, $L \subseteq X$. 

    The key fact is that $L$ is $\shij$-invariant.
    Hence $\shij(X) \supseteq \shij(L) = L \ni \coord(S)$, contradicting
    the choice of $S$.
  \item The space $\{ (t,\exp(te_{ij})\cdot x) : t\in \AA^1, x\in X\}$ 
    is isomorphic to $\AA^1 \times X$, hence irreducible. So its closure
    $\wt\swij(X)$ is irreducible, and the image $\swij(X)$ of that 
    is irreducible.
    
    To show the projection is degree $1$, it suffices to find a point
    in $\swij(X)$ over which the map is an isomorphism.
    By the non-invariance assumption, there exists a subset $M$ 
    such that $\coord(M \cup \{i\}) \in X, \coord(M \cup \{j\}) \notin X$.
    Hence the map $\wt\swij(X) \onto \swij(X)$ takes
    $$ (0, \coord(M \cup \{i\})
    \qquad\mapsto\qquad \coord(M \cup \{i\}). $$

    Let $D$ be the Pl\"ucker divisor $\{p_{M\cup \{j\}} = 0\}$. 
    By lemma \ref{lem:Plucker}, $D \supseteq X$.
    Since $D$ is codimension $1$ in $\Grkn$, $\wt\swij D$ is codimension $1$
    in $\PP^1 \times \Grkn$, and is easily seen to lie in the hypersurface
    $$ H := \left\{ ([t,u], V) \in \PP^1 \times \Grkn \ :\
      t p_{M \cup \{i\}} + u p_{M \cup \{j\}} = 0 \right\}. $$
    So far we have the maps
    $$ \{(0,\coord(M\union \{i\}))\} \into \wt\swij(X) 
    \into \wt\swij(D) \into H \onto \Grkn. $$
    Now we claim that the fiber over $\coord(M \cup \{i\})$ of
    this last projection $H \onto \Grkn$ is already a reduced point
    (namely, $(0,\coord(M\union \{i\}))$); it is
    $$ \{ ([t,u], \coord(M\union \{i\})) : t p_{M \cup \{i\}} = 0 \} $$
    and the remaining projective coordinate $p_{M \cup \{i\}}$ does not vanish. 

    Hence the fiber $\coord(M \cup \{i\})$ of $\wt\swij(X) \to \Grkn$
    is a reduced point. 
  \end{enumerate}
  \vskip -.25in
\end{proof}

\junk{  For $S \in {n\choose k}$,
  $$ \sh_{i\to j} \ \{\coord(S) \} = \coord(\sh_{i\to j} S). $$
  The left $\sh$ is the geometric shift, 
  the right $\sh$ the combinatorial shift.
}

Finally we are ready to give a theorem that can, in certain cases, 
calculate shifts. (While we hoped to use it with proposition
\ref{prop:matchings} to compute the shifts in theorem \ref{thm:safeshiftintro},
we will instead use proposition \ref{prop:safeshiftupperbound}.)

\begin{Theorem}\label{thm:shiftverify}
  Let $X,X'_1,\ldots,X'_m \subseteq \Grkn$ be $T$-invariant
  subvarieties of the same dimension. 
  Assume that $ \coord(\shij \coord^{-1}(X^T)) = \Union_{i=1}^m X'_m $ as schemes,
  and $\forall i$, $(X'_i)^T \setminus \Union_{j\neq i} (X'_j)^T \neq \emptyset$.
  Then $\shij(X) = \Union_{i=1}^m X'_m$.  
\end{Theorem}

\begin{proof}
  Since $X$ is reduced, 
  $$  \begin{array}{rclr}
    X &\subseteq& \bigcap_{S\in \calC} \left\{ V \in \Grkn : p_S(V)=0 \right\}
    & \text{by lemma \ref{lem:Plucker}, and} \\ 
    &&&\\
    \shij(X) &\subseteq& 
    \bigcap_{S \in \shji \calC} \left\{ V \in \Grkn : p_S(V)=0 \right\}
    & \quad \text{by proposition \ref{prop:shiftCback}.} 
  \end{array} 
  $$
  We know the right side by assumption, with the result that 
  $\shij(X) \subseteq \Union_{i=1}^m X'_m$. 

  Since $X$ is irreducible, it is equidimensional, so its flat limit
  $\shij(X)$ is set-theoretically equidimensional\footnote{%
    This is a standard application of Zariski's Main Theorem.  If
    $\shij(X)$ contained a geometric component $C$ of some smaller
    dimension $j$, we could cut the family $\wt\swij(X)$ down with
    $\PP^1\times P$, where $P$ is a general plane of codimension $j$,
    and discover that the still-irreducible $X \cap P$ degenerates to
    $\shij(X)\cap P$. But the latter contains isolated points $C\cap P$, 
    contradicting the connectivity guaranteed by ZMT.}  
  of that same dimension. Being contained in $\Union_{i=1}^m X'_m$, 
  it must be (as a set) a union of some of these components.

  Assume some $X'_i \not\subseteq \shij(X)$. By our second assumption,
  there exists a coordinate subspace 
  $V \in (X'_i)^T \setminus \cup_{j\neq i} (X'_j)^T$. 
  Hence $V \notin \shij(X)$. 
  This contradicts part (3) of theorem \ref{thm:Tconvex}, so establishing
  the opposite containment $\shij(X) \supseteq \Union_{i=1}^m X'_m$. 
\end{proof}

\subsection{Safe shifts of positroid varieties}\label{ssec:safeshifts}

Recall that $[a,b]$ denotes either the interval (if $a\leq b$) 
or the cyclic interval $[a,n] \cup [1,b]$ (if $a>b$).

\junk{

Call a shift $\shij$ \defn{a safe shift for the interval $[a,b]$} if
$\shji [a,b]$ (backwards shift!) is again an interval. 
Equivalently, if $i<j$ and $\shij$ is {\em not} safe for $[a,b]$, then
$i < a \leq j \leq b$, $a<b$, and $(a,b) \neq (i+1,j)$. 
There are three ways for it to be safe; $[a,b]$ might be $\shij$-invariant,
the \defn{backward safe shift} $\shji [i+1,j] = [a,j-1]$, 
and the \defn{forward safe shift} $\shji [j,i-1] = [j+1,i]$.

}

In the rest of this section 
\begin{itemize}
\item $i<j$ are fixed,
\item $\shij$ is a \defn{(nontrivially) safe shift for $J$}, meaning
  \begin{itemize}
  \item $J$ is a bounded juggling pattern of length $n$,
  \item $(i+1,j)$ is a crucial box for $J$ (giving a backward safe shift).
  \item all other crucial intervals are $\shij$-invariant, and
\end{itemize}
\item $J' = J \circ (j\leftrightarrow f(i))$.
\end{itemize}
(A \defn{trivially safe shift} is one for which {\em all}
the crucial intervals are $\shij$-invariant.)

We will compute the sweep $\swij \Pi_J$ of $\Pi_J$, and within that, the shift
$\shij \Pi_J$.

\begin{Lemma}\label{lem:safesweepcover}
  $J' \ \lessdot\ J$ is a covering relation in affine Bruhat order,
  i.e., $\Pi_J$ is a divisor in $\Pi_{J'}$.
\end{Lemma}

\begin{proof}
  Since $(i+1,j)$ is a crucial box, it is in the diagram, so not
  crossed out from above. Hence the dot in column $j$ of the affine
  permutation matrix is strictly below row $i$. Then we can construct
  the affine permutation of $J\circ (j\leftrightarrow f(i))$
  thusly: move $J$'s dot in column $j$ up to row $i$, and the old dot in
  row $i$ down to the now-empty row.

\begin{center}
  {
    \begin{minipage}[t]{0.6\linewidth}
  Since $(i+1,j)$ is crucial, we know $(i,j)$ is not in the diagram, 
  and must be crossed out from the right (not from above, or else $(i+1,j)$
  would be crossed out too). Hence the dot moving down is to the right
  of the dot moving up, $i < f^{-1}(j) \leq j < f(i)$.
  Therefore $J\circ (j\leftrightarrow f(i)) < f$ in affine Bruhat order.
    \end{minipage} }
  \hfill
  \raisebox{-1.2in}
  {
    \begin{minipage}[t]{0.30\linewidth}
      {\epsfig{file=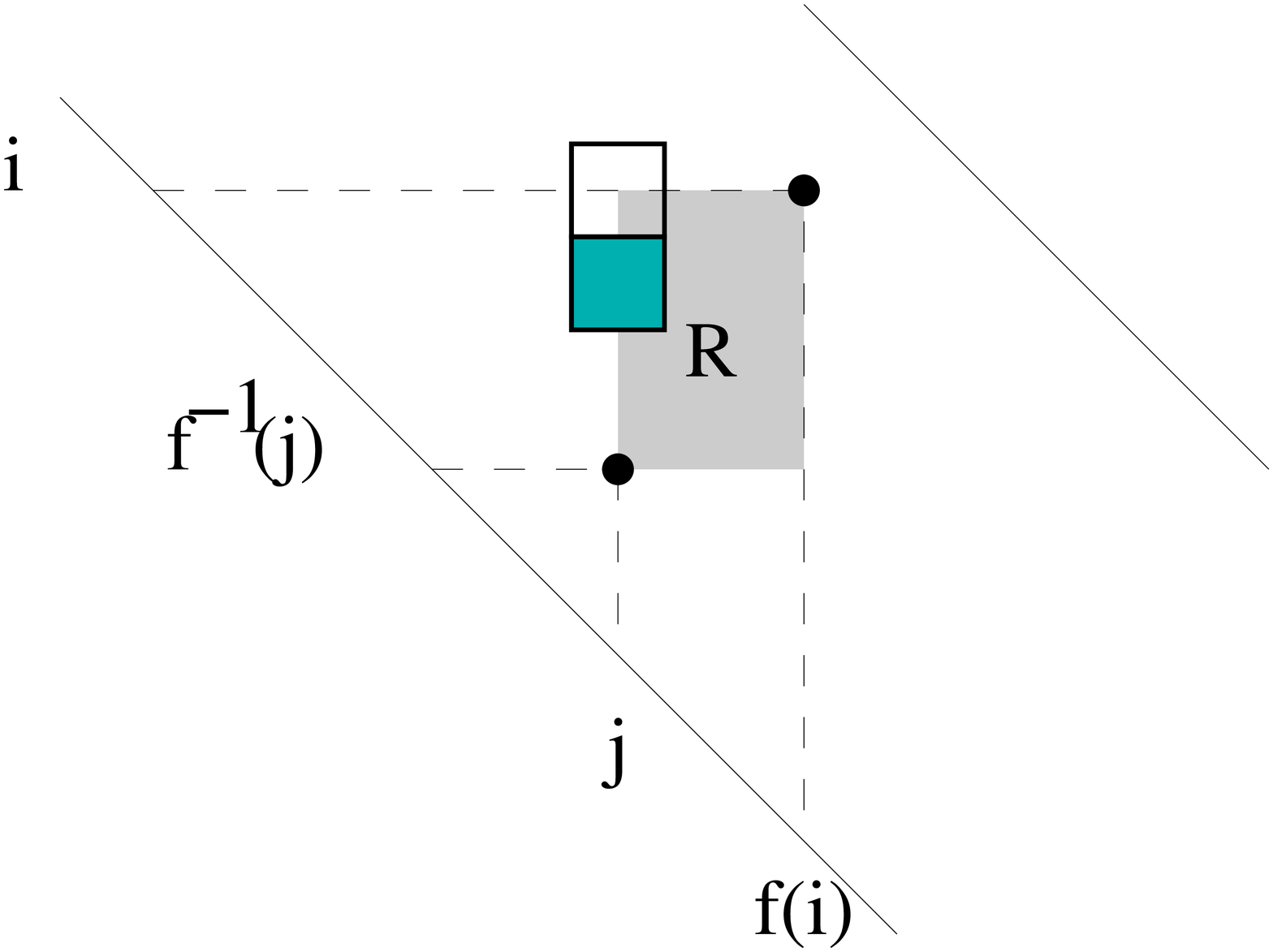,height=1.5in}}
    \end{minipage} }
\end{center}

  To show this is a covering relation, we need to show there are no other dots 
  in the rectangle $R$ within rows $(i,f^{-1}(j))$ and columns $(j,f(i))$.
  In fact we will show there are none to the right of $R$, either.

  Since $(i+1,j)$ is crucial, the box $(i+1,j+1)$ to its right
  is crossed out from above, so the entire second column of $R$ is crossed out.
  We learned before that the top left corner $(i,j)$ of $R$ is crossed out.
  But the bottom left corner $(f^{-1}(j),j)$ is not (since it contains a dot); 
  go up from it inside
  the diagram to find an essential box $(r,j)$. (It will automatically
  be crucial, otherwise $[r,j]$ would contain a crucial interval $[r',j]$
  with $r'>r$, but this would then give a lower essential box inside $R$
  and contradict the choice of $r$.)
  
  If there were dots to the right of $R$ in rows $[i+1,f^{-1}(j)-1]$,
  they would cross out more of $R$'s left column, and we would have
  $r>i+1$. The interval $[r,j]$ wouldn't be $\shij$-safe unless $r=j$.
  But then $[i+1,j]$ wouldn't be crucial; it would split as
  $[i+1,j-1] \coprod [j,j]$. 
\end{proof}

\begin{Proposition}\label{prop:safesweep}
  The safe sweep $\swij(\Pi_J)$ is again a positroid variety, $\Pi_{J'}$.

  Let $r$ be the rank bound on $[i+1,j]$ in the definition of $\Pi_J$. 
  Then $\Pi_J = \Pi_{J'} \cap X_{[i+1,j] \leq r}$ and
  $\shij(\Pi_J)\subseteq \Pi_{J'} \cap\ X_{[i,j-1] \leq r}$.
\end{Proposition}

\begin{proof}
  To apply proposition \ref{prop:sweep}, 
  we confirm that $\Pi_{J'}$ is $\shij$-invariant, and $\Pi_J$ isn't.

  Let $R$ again denote the rectangle with rows $[i,f^{-1}(j)]$ and columns 
  $[j,f(i)]$,
  and $R^\circ$ the subrectangle missing the outer rows and columns.
  We know the following about the diagrams of $J$ and $J'$ in $R$ and $R^\circ$:
  \begin{enumerate}
  \item The diagrams agree on $R^\circ$, 
    where (by the proof of lemma \ref{lem:safesweepcover}) they consist of
    entire columns of $R^\circ$.
  \item The $J$ diagram contains the column segment immediately to the left of 
    $R^\circ$, and not the one immediately to the right. 
    In $J'$ the opposite is true.
  \item The $J$ (resp. $J'$) diagram doesn't contain the row segment 
    immediately above (resp. below) $R^\circ$.
  \item The row segment below (resp. above) $R^\circ$ in the
    diagram of $J$ (resp. $J'$) is a continuation of the columns in $R^\circ$.
  \end{enumerate}

  The rank conditions of $J$ and $J'$ agree except on $R$ 
  (and do agree on the top row and right column).  By (4) above, the only
  possible essential conditions of $J'$ inside $R$ are on the top row,
  so on an interval $[i,s]$, and any such is $\shij$-invariant. 
  Hence $\Pi_{J'}$ is defined by $\shij$-invariant conditions,
  and is thus $\shij$-invariant.

  Similarly, the only possible essential boxes inside $R$ of $J$ 
  are in the second row, $(i+1,s)$, $s\geq j$. None of those with $s>j$ 
  can be crucial, or else $\shij$ would be unsafe for $\Pi_J$. 
  Hence $\Pi_J = \Pi_{J'} \cap X_{[i+1,j] \leq r}$ 
  where $r$ is the rank bound on $[i+1,j]$ in the definition of $\Pi_J$, and
  $$
  \shij \Pi_J \ =\ \shij (\Pi_{J'} \cap X_{[i+1,j] \leq r}) 
  \ \subseteq\ \shij \Pi_{J'} \cap \shij X_{[i+1,j] \leq r}
  \ =\ \Pi_{J'} \cap X_{[i,j-1] \leq r}.
  $$
  by proposition \ref{prop:shiftcupcap} and lemma \ref{lem:shiftSback}. 

  If $\Pi_J$ were $\shij$-invariant, then we would have 
  \begin{eqnarray*}
    \Pi_J &\subseteq& \Pi_{J'} \cap X_{[i+1,j] \leq r} \cap X_{[i,j-1] \leq r} 
    \ =\ \Pi_{J'} \cap (X_{[i+1,j-1]\leq r-1} \cup X_{[i,j]\leq r}) \\
    &=& (\Pi_{J'} \cap X_{[i+1,j-1]\leq r-1}) \cup (\Pi_{J'} \cap X_{[i,j]\leq r})
    \qquad\qquad\text{as sets}
  \end{eqnarray*}
  but each of those intersections has codimension $>1$ inside $\Pi_{J'}$.
\end{proof}

In principle, to compute the shift we could use theorem \ref{thm:shiftverify}.
But we can give a more efficient calculation using our knowledge of
the poset of positroid varieties. First we study the upper bound
provided in proposition \ref{prop:safesweep}.

\begin{Proposition}\label{prop:safeshiftupperbound}
  Let $j'$ run over the columns of dots in $J'$
  that are minimally Northwest of $(i,J(j))$. 
  Each $j' \geq i$ gives a component $\Pi_{J'\circ (j\leftrightarrow j')}$ 
  of $\Pi_{J'} \ \cap\ X_{[i,j-1] \leq r}$, 
  each of codimension $1$ in $\Pi_{J'}$,
  and these are all the components.

  Also, $[\Pi_J] = \sum_{j'} [\Pi_{J' \circ (j\leftrightarrow j')}]$
  as elements of $H^*(\Grkn)$.
\end{Proposition}

\begin{proof}
  The first is a direct application of theorem \ref{thm:monk} (whose $J$ 
  is our $J'$). 

  For the second, we use theorem 7.1 of \cite{KLS} to assert that the
  cohomology classes of positroid varieties are representable using
  affine Stanley symmetric functions of their affine permutations.

  These functions enjoy a ``transition formula'' \cite[theorem 7]{LSLittle}
  $$ \sum_{w \gtrdot v,\ w = v\circ (r\leftrightarrow s),\ r>s} F_w 
  = \sum_{u \gtrdot v,\ u = v\circ (r\leftrightarrow s),\ r<s} F_u. $$
  When $v = J'$, the safeness assumptions ensure that 
  the left sum is just $F_J$.

  (It seems likely that there should be a more geometric proof,
  using theorem \ref{thm:shiftverify}, perhaps using 
  proposition \ref{prop:matchings} to find the required points
  separating the components.)
\end{proof}

\begin{Theorem}\label{thm:safeshift}
  Under the assumptions from the beginning of \S \ref{ssec:safeshifts},
  $$ \shij(\Pi_J) \ = \  \Union_{i'} \Pi_{J'\circ (j\leftrightarrow j'} $$
  where $j'$ runs over the columns of dots in $J\circ (i\leftrightarrow j)$
  that are minimally Northwest of $(i,J(j))$ and in columns $\geq i$.

  If $S = (j_1',j_2',\ldots,j_m')$ is a sublist of these $\{j'\}$, 
  with dots ordered Northeast/Southwest, then
  $$ \bigcap_{i'\in S}  \Pi_{J'\circ (j\leftrightarrow j')} 
  = \Pi_{J\circ (i\leftrightarrow j)\circ 
    (i_1' \mapsto i_2' \mapsto \cdots \mapsto i_m' \mapsto i \mapsto i_1')}. $$
  In particular, as $K_T$-classes,
  $$ \big[\shij(\Pi_J)\big] \ = \ \sum_{S \neq \emptyset} (-1)^{|S|-1}
  \big[ \Pi_{J\circ (i\leftrightarrow j)\circ 
(i_1' \mapsto i_2' \mapsto \cdots \mapsto i_m' \mapsto i \mapsto i_1')} \big].$$
\end{Theorem}

\begin{proof}
  The containment $\subseteq$ comes from propositions \ref{prop:safesweep}
  and \ref{prop:safeshiftupperbound}. 
  As in theorem \ref{thm:shiftverify}, since the shift is equidimensional, 
  it must set-theoretically be a union of some of these components.
  But if some components were missing, then the homology classes would
  not match as in proposition \ref{prop:safeshiftupperbound}.

  The computation of $\bigcap_{i'\in S} \Pi_{J'\circ (j\leftrightarrow j')}$
  is essentially an inflated version of the following one:
  if $r_1 r_2 \cdots r_{|C|}$ is a Coxeter element, 
  and $S \subseteq \{1,\ldots,|C|\}$, then the Schubert varieties
  inside the Coxeter Schubert variety $X^{r_1 r_2 \cdots r_{|C|}}$ satisfy
  $$ \bigcap_{i \in S} X^{r_1 r_2 \cdots \widehat{r_i} \cdots r_{|C|}}
  = X^{\prod_{i \notin S} r_i}. $$
  That scheme-theoretic statement, 
  plus the trivial M\"obius inversion on the boolean lattice of
  subsets $S \subseteq C$, give the $K_T$-class formula.
\end{proof}

This and proposition \ref{prop:KTshift} suffice to prove
theorem \ref{thm:safeshiftintro}, which was less specific in
not making precise the components.

\section{The main theorems:
  IP pipe dreams as a record of shifting}\label{sec:main}

\subsection{The partial permutation matrix associated to a viable slice}
\label{ssec:viable}

Let $s$ be a slice at $(i,j)$, as pictured back in figure \ref{fig:slice}
from \S \ref{ssec:main}.
We will attempt to associate a partial permutation to $s$, and if 
we are successful we will call $s$ ``viable''. 

Call the area of the upper triangle that is above $s$, the \defn{top half},
and the remainder the \defn{bottom half}. 
Call the $(i,j)$ box the \defn{kink} in $s$.
Draw \defn{rays} (as in figure \ref{fig:slicedots})
perpendicular to the edges of $s$, as follows:
\begin{itemize}
\item The $0$ edges have rays pointing South or East, so {\em out}
  of the top half.
\item All other edges have rays pointing North or West, so typically
  {\em into} the top half.
\end{itemize}
(If $i=1$ some of the slice edges are on the top line, 
from which neither North nor South rays go into the top half.)

The rays are labeled with their edge label, much like the pipes are
in the pipe dreams.

\subsubsection{In the absence of $K$-labels}

For each letter label, say $A$, consider the vertical rays (going North
from an $\hlabel A$) 
and horizontal rays (going West from an $\vlabel A$) that are labeled $A$.

In the absence of $K$-tiles, we say that $s$ is \defn{viable} if
\begin{itemize}
\item for each letter label $A$, 
  the number of $\hlabel A$ and $\vlabel A$ edges in $s$ agree, 
\item for each $i$ up to that number, the $i$th $\hlabel A$ from the left 
  occur further South (and of course West) 
  than the $i$th $\vlabel A$ from the top, and
\item if there is a $\vlabel 1$ (necessarily on the East edge of the kink),
  there should also be a $\hlabelone$ on the bottom edge of the slice.
\end{itemize}

When these hold, we can place \defn{the $i$th $A$ dot} where 
the ray up from the $i$th $\hlabel A$
and the ray left from the $i$th $\vlabel A$ meet.
If there is a $\vlabel 1$, make its West-pointing ray meet the {\em rightmost}
ray up from a $\hlabelone$ to make the \defn{$1$ dot}. 
If we terminate the rays at those dots, as in figure \ref{fig:slicedots},
then no two rays with the same label cross.

There are also some \defn{$0$ dots} in the bottom half. Put a $0$ ray pointing
East through the triangle, in every row after the $i$th. If there are
$m$ South-pointing rays from $\hlabelzero$ edges, have them terminate
where they cross the top $m$ East-pointing $0$ rays (including one
from the East edge of the kink, if it is a $\vlabel 0$).
An example is in figure \ref{fig:slicedots}.

\begin{figure}[htbp]
  \centering
  \epsfig{file=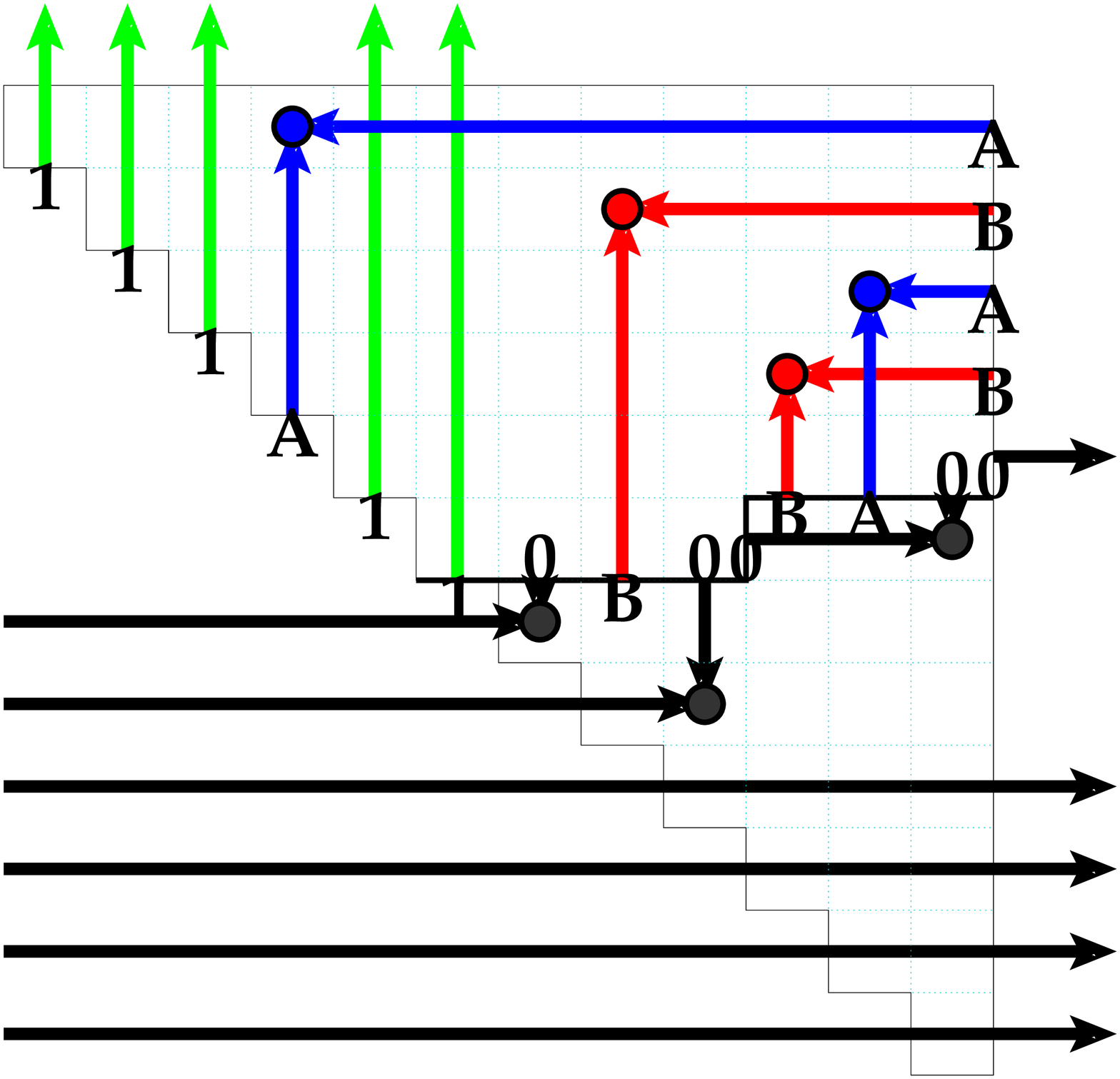,height=3.5in}
  \caption{The rays and dots of a slice. This example has no $1$-dot,
    because the East edge of the kink is not $\vlabel 1$.}
  \label{fig:slicedots}
\end{figure}

So in the absence of $K$-labels, the above arrangement of dots is
our definition of \defn{the partial permutation associated to the slice $s$}.
We will denote this $g(s)$, later to avoid confusion with the $f(P)$
associated to an IP pipe dream.

If one extends $f$ to a bounded juggling pattern $J$,
by placing dots Northwest/Southeast in the missing rows and columns,
then the $0$- and $1$-rays that continue outside the triangle can be 
imagined as pointing at these other dots.

\subsubsection{With $K$-labels}\label{sssec:wKlabels}
So far the rays described have only one label, and two rays either
cross unimpeded if they have different labels, or mutually annihilate
(leaving a dot) if they have the same label.

What changes now is that there is one slice edge that can have more
than one label: the East edge of the kink, labeled $V\not\ni 0$. 
We give this West-pointing ray the special property that if it crosses
a North-pointing ray with (just one) label $c$, then
\begin{itemize}
\item if $c$ is not in $V$, and $V$ doesn't end with $1$, the 
  rays cross unimpeded
\item if $c$ is not in $V$, and $V$ ends with $1$, $c$ must be $0$ and the 
  rays cross unimpeded
\item if $V$ ends with $c$, then the rays cross through each other
  {\em but both change} as depicted in the displacer tile; in particular
  the West-pointing ray loses its terminal letter ($c$)
\item if $c$ is in $V$, but is not its last letter, 
  then $s$ \defn{is not viable}.
\end{itemize}
To check viability, then, we continue this West-pointing ray, 
successively losing its terminal letters where it doesn't cross 
disjointly-labeled rays, until it gets down to a single label. 
At that point we have reduced to the previous definition of viability.


\begin{Proposition}\label{prop:safe00essential}
  Let $s$ be a viable slice, with kink at $(i,j)$, 
  and $f$ its associated partial permutation.
  Then $\shij$ is a safe shift for $\Pi_f$.

  The East and South edges of the kink are both labeled $0$
  iff $(i+1,j)$ is an essential box for $f$.
  Otherwise, $\Pi_f$ is $\shij$-invariant.
\end{Proposition}

\begin{proof}
  The $i=j$ case is silly and we dispense with it first. 
  Of course the shift is safe, $\Pi_f$ is $\sh_{i\to i}$-invariant, 
  and the South edge is labeled $1$, not $0$.

  Now, let $(a,b)$ be an essential box in $f$'s diagram, 
  with $i < a \leq j \leq b$; we need to show that $(a,b) = (i+1,j)$.  

  Since $a\geq i+1$, the box $(a,b)$ is in the ``bottom half'', where
  there are only $0$-dots.
  Since $(a,b)$ is not crossed out from the East
  (meaning, in $f$'s diagram -- not by a ray, in this context!) 
  there is a $0$-dot to its West.
  Assume for contradiction that $a > i+1$.
  Then there is also a $0$-dot in the row above, further West.
  Since $(a,b)$ is not crossed out from above, 
  neither is the box $(a-1,b)$ above it. 
  So if $a>i+1$, then $(a-1,b)$ is not crossed out at all, 
  making $(a,b)$ inessential. Contradiction; hence $a=i+1$.

  Since $(a=i+1,b)$ is not crossed out from above, neither is $(i,b)$,
  so it must be crossed out from the East (or else $(i+1,b)$ wouldn't
  be essential). Hence the East edge of the kink must be $\vlabel 0$,
  and the ray coming East out of this $\vlabel 0$ must pass all the
  way through the $(i,b)$ box.

  If $b>j$, then since $(i,b)$ is not crossed out from above, the
  slice label atop the $(i,b)$ square must be $\hlabelzero$. So the
  ray just mentioned would stop in the $(i,b)$ square, 
  not pass through, contradiction. 
  Hence $j=b$, concluding the proof that the shift is safe.

  Say $(i+1,j)$ is essential. Then since it is not
  crossed out from above, the slice label South of the kink
  must be $\hlabelzero$. Since the box above must be crossed out from the East,
  the slice label East of the kink must also be $\vlabel 0$. 

  Now the converse.
  If the East edge of the kink is $\vlabel 0$, then $(i,j)$ is not in
  the diagram (either the row's dot is to the right, or there is no dot). 
  If the South edge of the kink is $\hlabelzero$, 
  then $(i+1,j)$ is not crossed out from the North or East, and
  is in the diagram. For $(i+1,j)$ to be essential, though, we still need
  that $(i+1,j+1)$ is crossed out, and there are two cases to consider. 
  If the slice label above $(i,j+1)$ is $\hlabelzero$, then there is a
  $0$-dot at $(i,j+1)$. Otherwise there is a ray upward from this 
  slice label, and $(i,j+1)$ is crossed out from above. Either way
  $(i+1,j+1)$ is crossed out from above. This shows that if both labels
  are $0$, then $(i+1,j)$ is an essential box for $f$.
\end{proof}

\subsection{The West edges are automatically $0$s}

\begin{Proposition}\label{prop:west0s}
  Let $P$ satisfy all the requirements of a $K$-IP pipe dream
  except for the condition that the West edge labels are $0$s.
  Then this condition holds iff the number of letters on the
  South edges equals that on the East edges.  
\end{Proposition}

\begin{proof}
  If a West edge label of a tile contains a $1$, then the tile can't
  be a displacer (since the $Wbc$ on the East edge wouldn't end with $1$), 
  so the South edge must be $0$.  This can't happen on a tile at
  $(i,i)$, so none of the West edges of the pipe dream contain $1$s.

  If a $1$-pipe enters a fusor tile from the West, it must come out the North. 
  So the $1$-pipes go from South edges of the $K$-IP pipe dream to North edges.

\junk{

  Now we argue that every letter pipe coming from a South edge
  comes out an East edge. Briefly erase, from every vertical edge,
  all labels except for the rightmost. Then displacer tiles look like
  fusor tiles except that neither of $b$ nor $c$ is necessarily $0$.
  Regardless, the pipes are well-defined.
}

  Denote the numbers of labels on the North edge by $N_0 + N_1 = n$, 
  where $N_0,N_1$ are the numbers of labels $0,1$, on the East edge 
  by $E_0 + E_L = n$ (with $L$ for Letter), and on the South edge
  by $S_1 + S_L = n$. Then so far we have argued $S_1 = N_1$. 
  From the North, the number of $0$-pipes coming out the West is
  $n-N_1 = n-S_1 = S_L$. From the East, the number is $n-E_L$.
  Summing, we get $n-E_L+S_L$ $0$-pipes on the West, 
  so every West edge ends with $0$ iff $S_L = E_L$.
\end{proof}

\subsection{Placing the next tile}\label{ssec:placing}

Let $s$ be a viable slice, with kink at $(i,j)$. Say that 
$s$ \defn{admits} a tile $\tau$, \defn{producing} the slice $s'$, if
\begin{itemize}
\item the upper half of $s$ has one more box (namely, the kink) than
  the upper half of $s'$,
\item if $s,s'$ agree on all common edges, 
\item the edges on which they differ bound the tile $\tau$ (at $(i,j)$), and
\item $s'$ is again viable.
\end{itemize}

\begin{Proposition}\label{prop:uniquefill}
  Let $s$ be a viable slice such that the East and South edges of the
  kink are not both labeled $0$. Then $s$ admits a unique tile, and
  the $s'$ produced has the same associated partial permutation as $s$ did.
\end{Proposition}

\begin{proof}
  This is a straightforward case check, which we recommend to the reader.
  Spoilers commence for those who resist the pleasure.

\begin{center}
  {
    \begin{minipage}[t]{0.72\linewidth}

  If the East and South edge of the kink have disjoint labels, then
  the tile $\tau$ must be a crossing tile. The rays from the new
  horizontal and vertical edges are labeled and pointing the same
  directions as before.  They match up with (or otherwise modify) the
  same perpendicular rays as they did before.
    \end{minipage} }
  \hfill
  \raisebox{-0.7in}
  {
    \begin{minipage}[t]{0.20\linewidth}
      {\epsfig{file=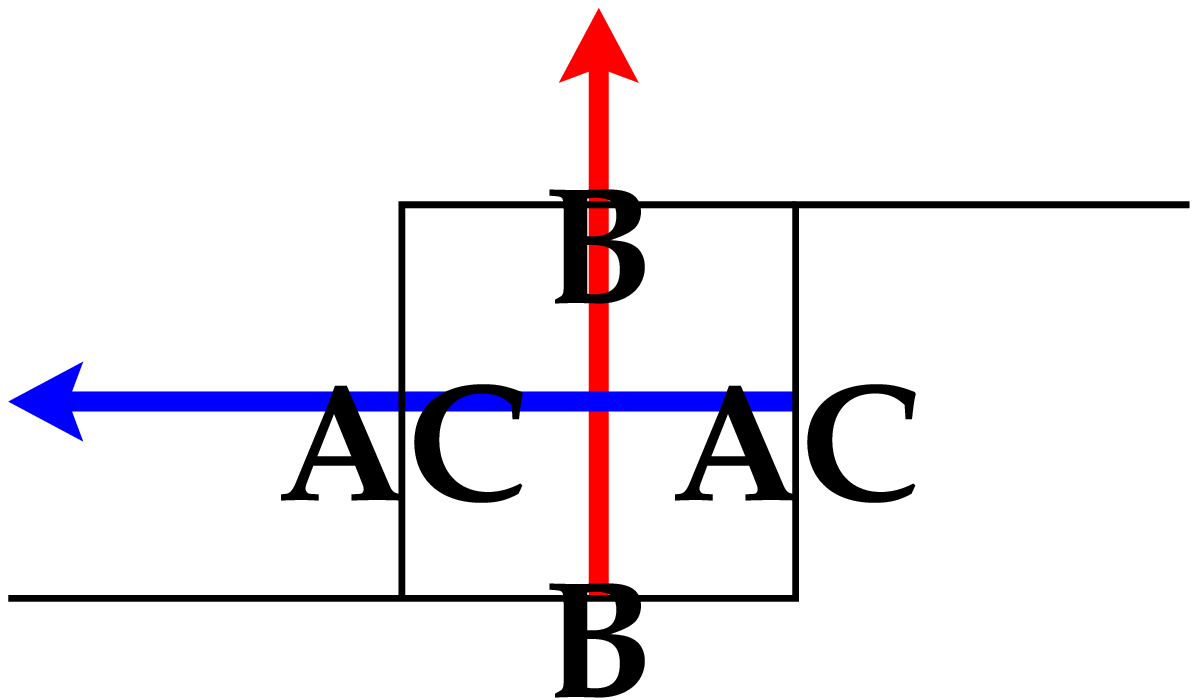,height=.8in}}
    \end{minipage} }
\end{center}

\begin{center}
  {
    \begin{minipage}[t]{0.20\linewidth}
      {\epsfig{file=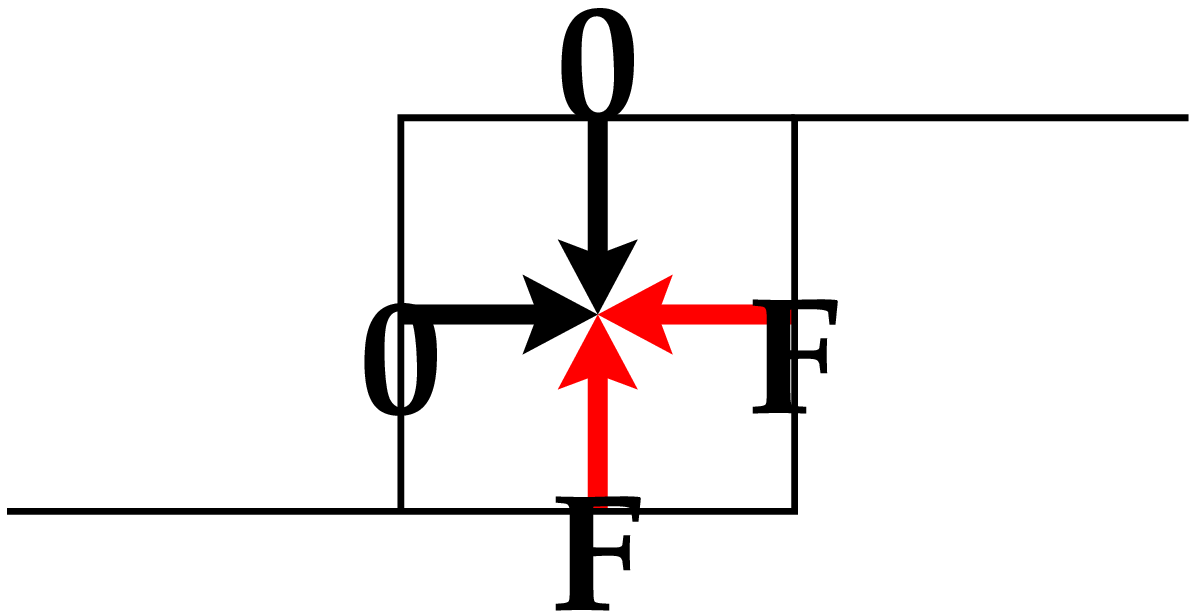,height=.8in}}
    \end{minipage} }
  \hfill
  \raisebox{0.7in}
  {
    \begin{minipage}[t]{0.72\linewidth}

  If these edges have the same label $b\neq 0$, then in the partial 
  permutation associated to $s$, there is a $b$ dot inside the kink.
  The tile $\tau$ must be a ``dot'' tile (in the list 
  from \S \ref{sssec:Ktiles}).
  When we fill it in, the slice $s'$ so produced has 
  a $0$ dot in the same place.
    \end{minipage} }
\end{center}

\begin{center}
  {
    \begin{minipage}[t]{0.72\linewidth}

  Otherwise the East edge must have multiple labels on it. 
  By the definition of viability from \S \ref{sssec:wKlabels},
  the South label must be the last letter of the East label,
  and filling in the displacer tile to create $s'$ both preserves
  viability, and leaves the dots in place. (This is of course
  due to our recursive definition of viability, 
  in the presence of multiple labels.)
    \end{minipage} }
  \hfill
  \raisebox{-0.7in}
  {
    \begin{minipage}[t]{0.20\linewidth}
      {\epsfig{file=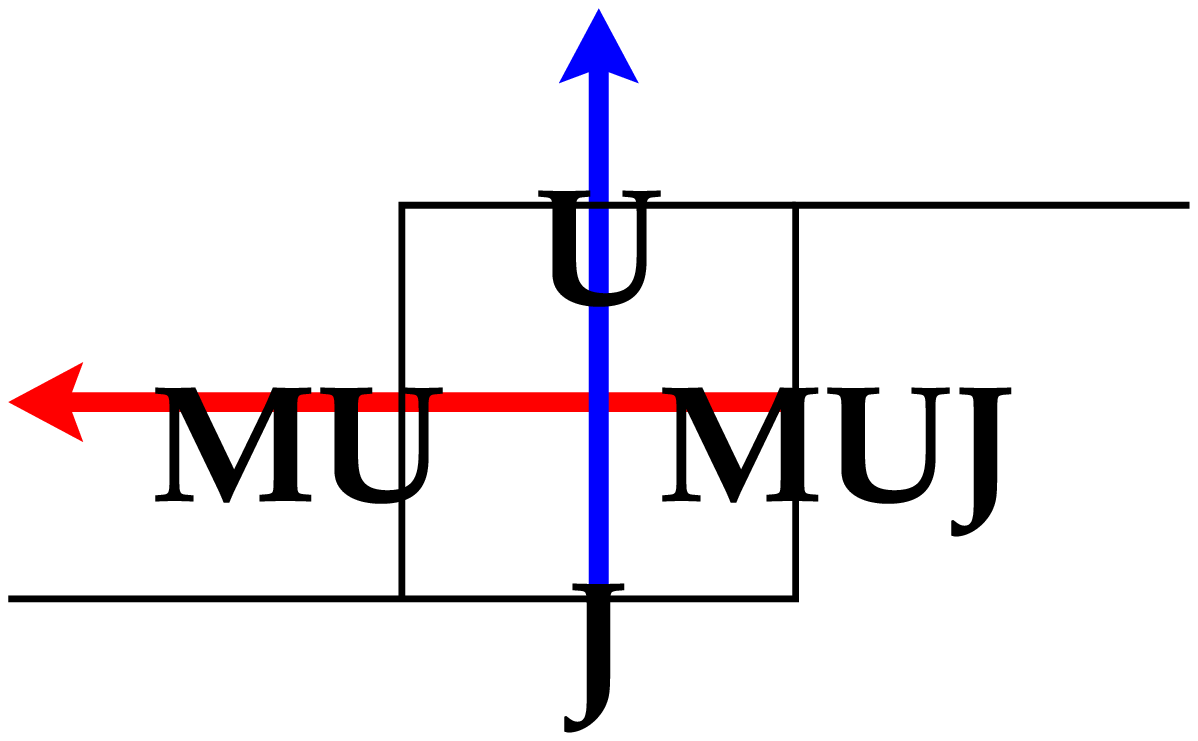,height=.9in}}
    \end{minipage} }
\end{center}
\vskip -.3in
\end{proof}

The remaining case -- when the East and South edges of the kink are
both labeled $0$ -- is much more interesting.

\begin{Proposition}\label{prop:nonuniquefill}
  Let $s$ be a viable slice, with kink at $(i,j)$, whose South and East edges 
  are labeled $0$. (In particular $i < j$, since there are no South
  $\hlabelzero$ labels on $(i,i)$ tiles.) So $s$ admits only fusor tiles.

  One such that $s$ admits is the equivariant tile, 
  and the $s'$ so produced has $\Pi_{f(s')} = \swij \Pi_{f(s)}$.

  The dots in $f(s)$ that are minimally Northwest of $(i,j)$,
  and in columns $[i,j]$, have all distinct labels. 
  Let $C$ be the list of these labels (read SW to NE),
  plus $1$ at the end if the $s$-labels West of the kink end $1\ 0^m$.

  For each nonempty sublist $S \subseteq C$, $s$ admits the fusor tile
  with West edge $S$. These tiles (and the equivariant tile) are all
  the tiles $s$ admits, and for each such $S$ the resulting $s'$ has
  $\Pi_{f(s')} = 
  \bigcap_{i'\in S} \Pi_{f(s) \circ (i\leftrightarrow j)\circ (i'\leftrightarrow i)}$
  as last seen in theorem \ref{thm:safeshift}.
\end{Proposition}

\begin{proof}

  Before placing the equivariant tile, 
  there are $0$-rays coming South and East out of the kink.
  The South-pointing ray goes down to a $0$-dot $\delta$, and the East-pointing
  ray either meets a South-pointing $0$-ray $\rho$ at a $0$-dot, 
  or it exits the bottom half entirely. 
  Once we place the equivariant tile at the kink, producing $s'$, the $0$-rays 
  each start one step back, and now collide at a $0$-dot in $(i,j)$.  
  If there is a $\rho$, it no longer hits the East-pointing $0$-ray 
  from the kink, but continues down to the row where $\delta$ was.
  Effectively, $\delta$ has moved up into $(i,j)$, and the $0$-dot from
  row $i$ (if there is one) has moved down to $\delta$'s row. 
  This is exactly the sweeping action on dots 
  computed in proposition \ref{prop:safesweep}.
  The fact that the ray/dot picture continues to exist is our
  definition of viability.

  Since two dots in $f(s)$ with the same label 
  were required to be NW/SE of each other,
  the dots that are minimally Northwest of $(i,j)$ must have distinct labels.
  Let $C$ be the list of those dots in columns $[i,j]$ (read SW to NE),
  plus $1$ at the end if the $s$-labels west of the kink end $1\ 0^m$.
  We now claim that $s$ admits a fusor tile with West edge $W$
  $\Longleftrightarrow$ the list $W$ is a (nonempty) sublist of $C$.

  For each direction of this iff, it helps to understand what tiles will be 
  placed after (i.e. further left from) the fusor is placed. There are only
  crossings (which copy the vertical labels) and displacers (which
  remove one letter from $W$ at a time, from the right), until all
  the letters in $W$ are gone and we hit a dot tile. (More tiles are
  forced thereafter, usually, but this is enough to consider.) We
  have to hit a dot tile at some point, 
  because by proposition \ref{prop:west0s} the leftmost vertical
  edge will be a $\vlabel 0$.
  See figure \ref{fig:Krays} for the full story.

  \begin{figure}[htbp]
    \centering
    \epsfig{file=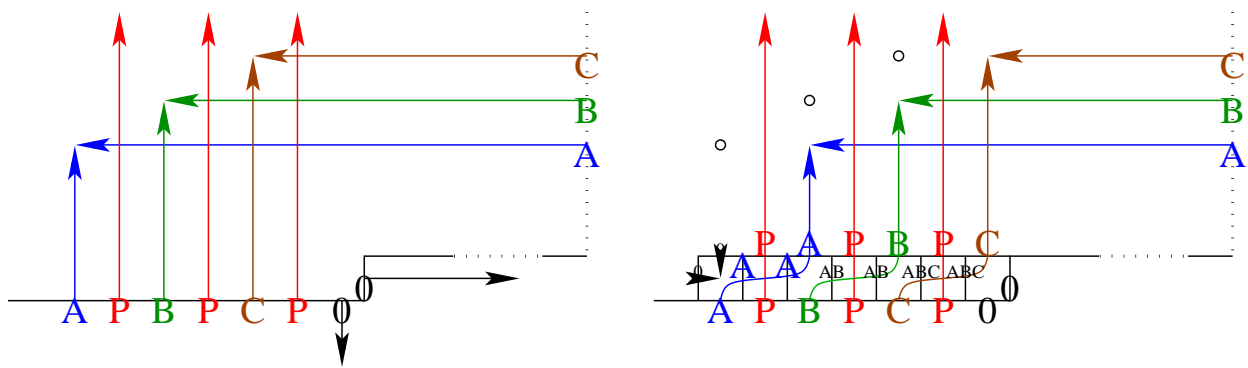,width=6.5in}    
    \caption{From the left to the right figure, we place the
      fusor tile, forcing the next several displacer and crossing tiles,
      then the dot tile.
      The little circles in the right figure show where the dots were, 
      before they moved rightward. Hopefully this picture
      also indicates the motivation of the name ``displacer'' tile.}
    \label{fig:Krays}
  \end{figure}

  $\Longleftarrow$ Place the tile, producing $s'$, and the follow-on
  tiles up through the first dot tile. To show this
  preserves viability, we have to show that there is again a consistent
  system of rays and dots; this is best illustrated in figure \ref{fig:Krays}.

  $\Longrightarrow$
  We claim that if $s$ admits a fusor tile with West edge $W$,
  then $W$ must be a sublist of $C$. Each displacer tile encountered
  before the dot tile lies over a letter or a $1$, 
  with a North-pointing ray. We claim that the dot that ray points to
  (interpreting the ray from a $\hlabel 1$ as pointing to a dot just
  outside the triangle) is minimally Northwest of $(i,j)$.

  This is easy for the $\hlabel 1$ case. By the condition on crossing tiles, 
  before we meet the $\hlabel 1$ displacer we go through 
  $\hlabelzero$ crossing tiles, whose dots are to the South.

  Otherwise the displacer tile at $(i,j')$ lies atop a letter, say $A$, 
  with a ray pointing up to an $A$-dot. If the $A$-dot is not
  minimally Northwest of $(i,j)$, then there is another (say) $B$-dot
  in some column $j''$, $j' < j'' < j$ in between the $A$-dot and $(i,j)$.
  (So $B$ is a letter, not $0$ or $1$.)
  We do assume this $B$-dot to be minimally Northwest, so, lying
  atop the rightmost $\hlabel B$ left of $(i,j)$.
  This letter $B$ must be distinct from $A$, or else the $A$ would
  already have been ripped out of the vertical label.
  Before we place the fusor tile, the $A$-rays and $B$-rays do not 
  intersect at all, by the $B$-dot being SE of the $A$-dot. 

  There are two cases: the fusor tile involves the label $B$, or not.

  If not, then once we place tiles at $(i,j)$--$(i,j')$, going from
  the left figure here to the right figure,

  \centerline{\epsfig{file=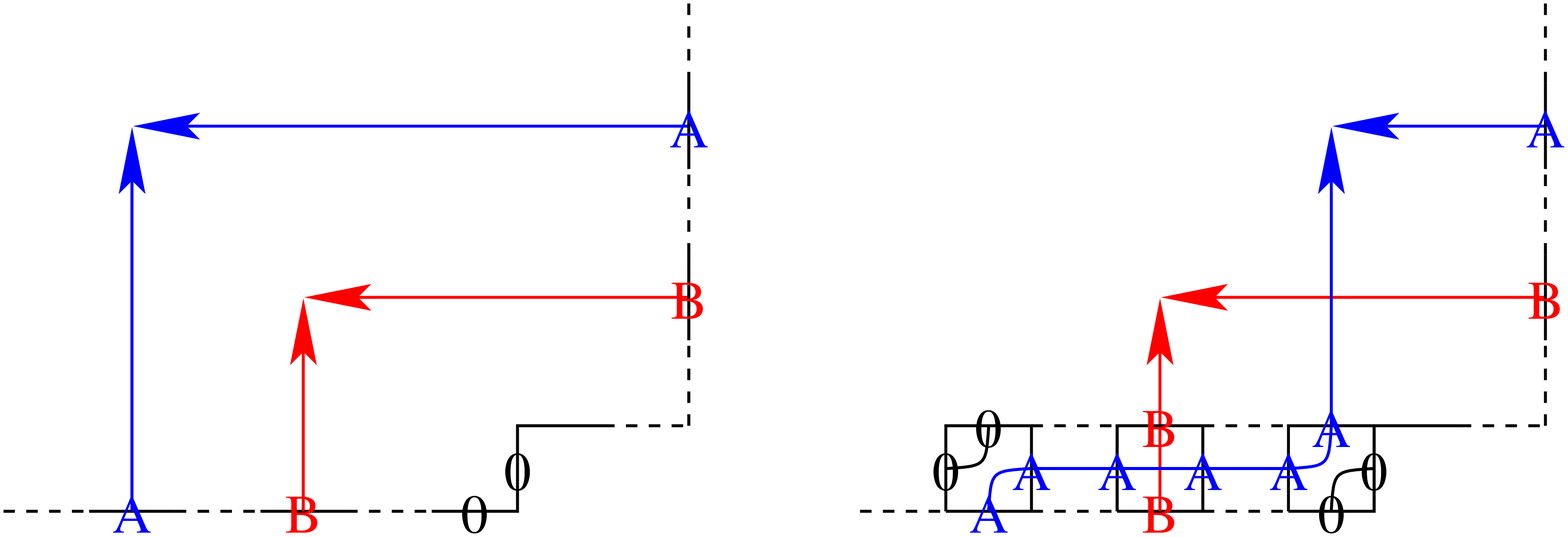,width=6.5in}    }

  \noindent the $A$-pipe crosses the $B$-pipe once in the tiles 
  and once in the rays. Placing more tiles won't fix the latter intersection, 
  by the Jordan curve theorem, so we know that whatever pipe dream we
  make eventually will have two lettered pipes crossing twice. 
  That being forbidden is the contradiction that says there is no
  offending $B$-dot.

  If yes, then the West label of the fusor tile involves both $A$ and $B$
  so we are definitely using $K$-pieces. Now we invoke the nonlocal
  condition on $K$-IP pipe dreams (look again at \S \ref{sssec:Ktiles})
  and Jordan curve to reach a similar contradiction.
\end{proof}

\subsection{Proofs of the main theorems}

\begin{proof}[Proof of theorem \ref{thm:main}]
  This just combines propositions 
  \ref{prop:safe00essential}, \ref{prop:uniquefill}, 
  and \ref{prop:nonuniquefill}.
\end{proof}

The other proofs are by induction through the Vakil sequence. 
For each pair $(i \leq j)$,
define a \defn{$K$-IP pipe dream $P$ below $(i,j)$} to be 
a viable slice at $(i,j)$ plus a filling of its bottom half with $K$-tiles.
We can interpret the previous 
definition of $K$-IP pipe dream as the $(i=0,j=n)$ case,
and otherwise require $i \in [n]$.
(These partial pipe dreams will be directly useful in \cite{KL}.)

It is clear how to extend the definition of $f(P)$, $fusing(P)$, and
$wt_K(P)$ to these partial pipe dreams. And since each comes with a slice,
we can define the partial permutation $g(P)$ to be $g($that slice$)$. 
For $(i,j) = (n,n)$ (no tiles) we have $g(P)=f(P)$, 
whereas for $(i,j) = (0,n)$ (usual $K$-IP pipe dreams),
we have $g(P)$ running NW/SE, in the first $n-k$ rows, 
with columns determined by $\lambda(P)$.

\begin{proof}[Proof of theorems \ref{thm:KTformula},
  \ref{thm:Kformula}, \ref{thm:HTformula}, \ref{thm:Hformula}]
  We will prove first a generalization of theorem \ref{thm:KTformula}, 
  that for each $(i \leq j)$,
  $$ [\Pi_f] = \sum_{P:\ f(P) = f} (-1)^{fusing(P)} wt_K(P)\ [\Pi_{g(P)}] $$
  where the $P$ summed over are the $K$-IP pipe dreams below $(i,j)$.
  When $(i,j) = (0,n)$, this is just the sum over $K$-IP pipe dreams $P$,
  and $\Pi_{g(P)} = X^{\lambda(P)}$ by lemma \ref{lem:oppSchubert}.

  The proof will be by induction through the Vakil order,
  where the base case is $(i,j)=(n,n)$, handled by
  lemma \ref{lem:Richardson}. There is a unique (tile-less)
  $K$-IP pipe dream $P$ below $(n,n)$, and $g(P) = f(P) = f$, 
  giving the equation $[\Pi_f] = (-1)^0 \cdot 1 \cdot [\Pi_f]$. 
  In the inductive step, we want to place one tile on each $P$
  in the summation. 

  If the South and East edges of the kink are not both labeled $0$,
  then $\Pi_{g(P)}$ is $\shij$-invariant 
  (proposition \ref{prop:safe00essential}), 
  and there is a unique way of placing the tile and it does not move
  the dots (proposition \ref{prop:uniquefill}).

  If the South and East edges of the kink are both labeled $0$,
  let $P_e$ be $P$ plus an equivariant piece, 
  and $\{P_S\}_{S \subseteq C}, S\neq \emptyset$, where $C$ is as in 
  proposition \ref{prop:nonuniquefill}. 

  Since positroid varieties are $T$-convex by theorem \ref{thm:Tconvex}
  (1 or 2), we can use the degree $1$ result from theorem \ref{thm:Tconvex} (3).
  That, and them having rational singularities
  justifies use of proposition \ref{prop:KTshift} which computes $K_T$-classes
  using shifts. It says
  \begin{eqnarray*}
     [\Pi_{g(P)}] &=& (1 - \exp(y_j-y_i)) [\swij \Pi_{g(P)}] 
     + \exp(y_j-y_i) [\shij \Pi_{g(P)}] \\
     &=&  (1 - \exp(y_j-y_i)) [\Pi_{g(P_e)}]
     + \exp(y_j-y_i) \sum_{S \subseteq C,\ S\neq \emptyset} (-1)^{\#S - 1} [\Pi_{g(P_S)}] \\
     && \text{by proposition \ref{prop:nonuniquefill} 
       and theorem \ref{thm:safeshift}, which we rewrite as} \\
     &&\\
     &=& \sum_{P\text{ admits }\tau} (-1)^{fusing(\tau)} wt_K(\tau) [\Pi_{g(P\& \tau)}].
  \end{eqnarray*}
  If we use that recursively, and the multiplicativity of the definition
  of $fusing$ and $wt_K$, we get the generalization claimed.

  This directly implies theorem \ref{thm:Kformula} by taking all $y_i = 0$,
  and from there theorem \ref{thm:Hformula} by taking leading terms.

  With some care we could derive the $H^*_T$ theorem \ref{thm:HTformula} 
  as the leading terms of this $K_T$ formula. But it is simpler just
  to replace the use of proposition \ref{prop:KTshift}'s $K_T$-formula
  with its $H^*_T$-formula.  From that point the derivation is the same.

  Finally, we address the fusing count from the end of 
  theorem \ref{thm:Kformula}, proving by induction more generally that
  $$ fusing(P) = \dim \Pi_{f(P)} - \dim \Pi_{g(P)} + \#\text{equivariant tiles} $$ 
  for these partial pipe dreams. The base case of the induction is 
  $P$ empty, where $f(P)=g(P)$. For the induction, we see how $g(P)$ changes
  as we attach one tile, using the analyses of propositions 
  \ref{prop:uniquefill} and \ref{prop:nonuniquefill}. There are three cases:
  \begin{enumerate}
  \item \emph{Unique fill.} See proposition \ref{prop:uniquefill}.
    We don't change $fusing(P)$ or $g(P)$, nor attach an equivariant tile,
    so none of the four numbers change.
  \item \emph{Equivariant tile.} See proposition \ref{prop:nonuniquefill}.
    We change $g(P)$ by a covering relation in affine Bruhat order,
    increasing $\dim \Pi_{g(P)}$ and the number of equivariant tiles 
    each by $1$.
  \item \emph{Fusing.} See proposition \ref{prop:nonuniquefill}.
    The change in length of $g(P)$ matches the change in the fusing number.
  \end{enumerate}
\end{proof}

\section{Puzzles $\quad\leftrightarrow\quad$ 
  IP pipe dreams with one letter}\label{sec:puzzles}

In this section we consider IP pipe dreams (not $K$-IP) with only one letter,
i.e. the edge labels are $\{0,1,R\}$.
This has two interesting effects.

The first is that in the initial (and every later)
slice $s$ the dots in the
top half of $f(s)$ are NW/SE, and consequently, the interval
positroid variety associated to an initial slice (where the 
top half is the whole upper triangle) is a Richardson variety
$X_\mu\cap X^\nu$.
In particular, its homology class is given by the Littlewood-Richardson rule, 
and its geometric shifts have already been studied in \cite{Vakil}.

The second effect is that the conditions defining the IP (but not $K$-IP)
pipe dreams become entirely local: since two lettered pipes with the
same letter don't cross even once, they definitely won't cross twice,
and there aren't two different letters to worry about.

\subsection{Puzzles}
We give a slightly modified definition of the 
{\em equivariant puzzles} from \cite{KT}.
The \defn{puzzle labels} are $0$, $1$, $R$, and
the \defn{puzzle pieces} are these:

\centerline{\epsfig{file=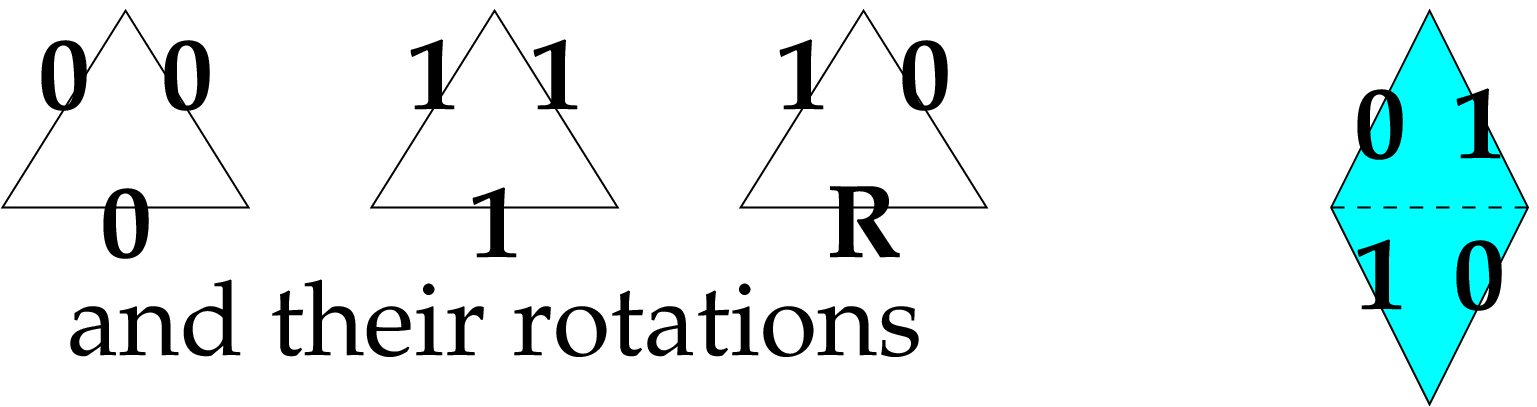,height=.9in}}

\noindent A size $n$ \defn{puzzle} is a tiling of a size $n$ equilateral 
triangle (parallel to those above)
with puzzle pieces, such that the boundary has no $R$ labels.
Consequently, the pieces with $R$s come in adjacent pairs, the $R$
being for ``rhombus''. The fourth piece is
the \defn{equivariant piece}, with \defn{equivariant weight} $y_i-y_j$,
where $i$ is the distance of that piece from the Northwest side of the
puzzle, and $n-j$ its distance from the Northeast side.
(Other pieces have equivariant weight $1$.)

To compare puzzles to IP pipe dreams, it is mnemonic to
first compare the equivariant rhombus of weight $y_i-y_j$ 
to the corresponding equivariant tile. We will need to stretch
the puzzle to drape each rhombus across the corresponding square, 
and will need to change the labels on {\em vertical} edges.

\junk{
The first step in comparing puzzles to IP pipe dreams is to stretch the
$60^\circ$-$60^\circ$-$60^\circ$ triangles to $45^\circ$-$45^\circ$-$90^\circ$,
so that a weight $y_i-y_j$ equivariant rhombus lands in the $(i,j)$ 
matrix entry; this is
by moving the North vertex Eastward and the Southwest vertex North.
The puzzle pieces are now

\centerline{\epsfig{file=puzpiecesbent.eps,width=6.5in}}

The second step is to notice that the equivariant puzzle piece has
the same horizontal labels as the equivariant tile, but not the
same vertical labels. So we need to convert those:
}

\begin{Theorem}\label{thm:vakilvarieties}
  Given a puzzle of size $n$, apply the following transformations:
  \begin{enumerate}
  \item Move the North corner right until it is above the Southeast corner,
    then the Southwest corner up until it is left of the North corner.
  \item Along the NW/SE diagonal, attach
    \epsfig{file=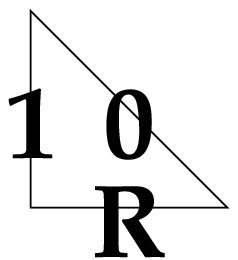,width=.4in}
    and
    \epsfig{file=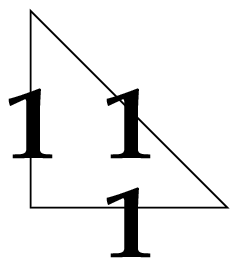,width=.4in} pieces
    so that the resulting shape is that of an IP pipe dream.
  \item Erase all diagonal labels, so the result is a bunch of labeled squares.
  \item On vertical edges, rotate the labeling 
    $\vlabel R \mapsto \vlabel 1 \mapsto \vlabel 0 \mapsto \vlabel R$.
    Horizontal labels we leave alone.
  \end{enumerate}
  Then the result is an IP pipe dream. 
  An example is in figure \ref{fig:puzIPex}.

  \begin{figure}[htbp]
    \centering
    \epsfig{file=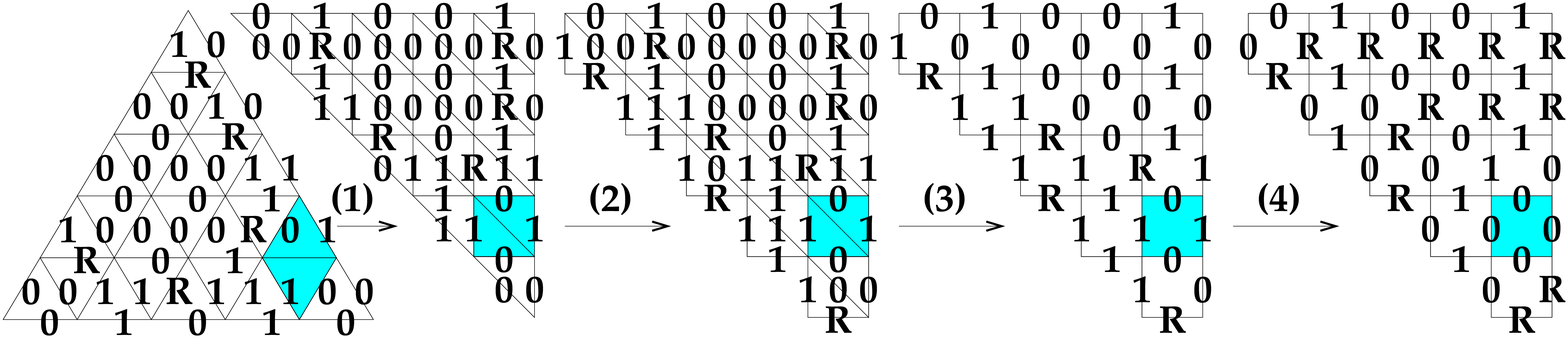,width=6.5in}
    \caption{The transformation of a puzzle, weight $y_4-y_5$,
      into an IP pipe dream.}
    \label{fig:puzIPex}
  \end{figure}

  Let $\lambda,\mu,\nu \in {[n]\choose k}$. 
  This composite transformation gives a weight-preserving bijection 
  between 
  \begin{itemize}
  \item puzzles with $\lambda,\mu,\nu$ giving the positions of the $1$s 
    on the NW, NE, and S sides (each read left-to-right), and
  \item IP pipe dreams with only the labels $\{0,1,R\}$,
    where $\lambda$ gives the positions of the $1$s on the North side, 
    and $\mu,\nu$ give the positions of the $R$s on 
    the East side (read top-to-bottom) and South side (read left-to-right).
  \end{itemize}
\end{Theorem}

\begin{proof}
  Since both definitions are local, the proof is simply a correspondence
  between the $9$ ways to attach two triangles together 
  ($1$ with $R$ on the diagonal, $2*2*2$ if one chooses the diagonal
  to be $0,1$ and then chooses each half), plus the equivariant rhombus, 
  to the $10$ possible tiles ($2$ dot tiles, $2$ fusors, $3*2-1$ crossing,
  $1$ equivariant). Here it is:

  \centerline{  \epsfig{file=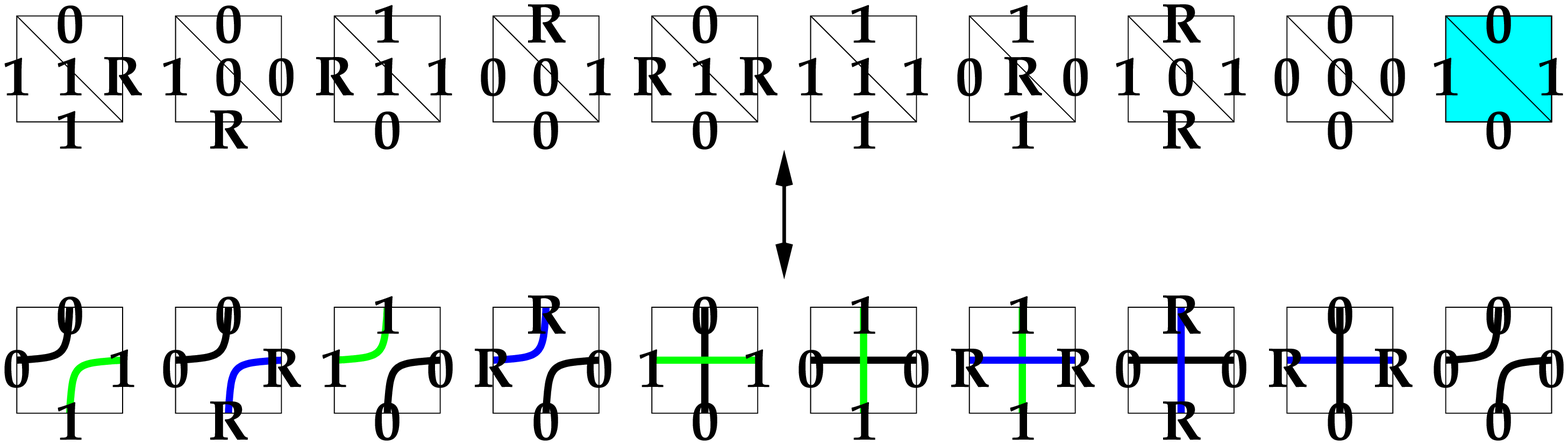,width=6.5in}}
\junk{
  \begin{itemize}
  \item $2$ dot tiles, with $0\& 1$ or $0\& R$,
  \item $2$ fusor tiles, with $1\& 0$ or $R\& 0$,
  \item $3*2-1$ crossing tiles (the $-1$ has $1$ horizontally crossing $R$), and
  \item the equivariant tile,
  \end{itemize}
}
\end{proof}

Since we are working in equivariant cohomology and not $K$-theory,
the Schubert and opposite Schubert bases are dual bases in the sense that
$\int_{\Grkn} [X_\lambda] [X^\mu] = \delta_{\lambda\mu}$. 
So instead of interpreting puzzles as computing the coefficient of
$[X^\lambda]$ in the class $[X_\mu \cap X^\nu] = [X_\mu][X^\nu]$ of the
Richardson variety, we can equivalently interpret them 
as computing the coefficient of $[X_\nu]$ in the product $[X_\lambda][X_\mu]$,
giving the main result of \cite{KT}.

One benefit of the puzzle combinatorics over that of the IP pipe dreams
is to make combinatorially evident a
$\integers/3$-symmetry of the nonequivariant Schubert structure constants,
namely $c_{\lambda \mu}^{\nu^*} = c_{\mu \nu}^{\lambda^*}$, 
since both equal $\int_{\Grkn} [X_\lambda] [X_\mu] [X_\nu]$.
(This symmetry does not hold equivariantly, since the dual basis
element $[X^{\lambda}]$ to $[X_\lambda]$ only equals $[X_{\lambda^*}]$ 
nonequivariantly. And sure enough, the puzzle rule is only 
rotationally symmetric when we exclude the equivariant piece.)

The paper \cite{PZJ} also reinterprets puzzles in terms of pipes, or
more properly space-time diagrams of two colors of fermions moving on a line, 
but the world-lines of those particles don't seem to have much to
do with the pipes here.


\begin{thebibliography}{xxxx}

\bibitem[ABS]{ABS} {\sc Sami Assaf, Nantel Bergeron, Frank Sottile:} 
  Multiplying Schubert polynomials by Schur functions, in preparation.

\bibitem[AGriMil11]{AGM} {\sc David Anderson, Stephen Griffeth, Ezra Miller:}
  Positivity and Kleiman transversality in equivariant $K$-theory
  of homogeneous spaces. J. European Math Society 13 (2011), 57--84.
  \url{http://arxiv.org/abs/0808.2785}

\bibitem[B76]{BB} {\sc Andrzej Bia\l ynicki-Birula:}
  Some properties of the decompositions of algebraic varieties 
  determined by actions of a torus. 
  Bull. Acad. Polon. Sci. S\'er. Sci. Math. Astronom. Phys. 24 (1976), 
  no. 9, 667--674. 

\bibitem[BiCo12]{BC} {\sc Sara Billey, Izzet Co\,skun:}
  Singularities of generalized Richardson varieties.
  Communications in Algebra. 40:4  (2012), 1466--1495.
  \url{http://arxiv.org/abs/1008.2785}

\bibitem[BdM08]{BdM} {\sc Joseph Bonin, Anna de Mier:}
  The lattice of cyclic flats of a matroid. 
  Ann. Comb. 12 (2008), no. 2, 155--170. 
  \url{http://arxiv.org/abs/math/0505689}

\bibitem[BGW03]{CoxeterMatroids} 
  {\sc Alexandre Borovik, Israel Gelfand, Neil White:}
  Coxeter matroids. Progress in Mathematics 216, Birkh\"auser.

\bibitem[Bri02]{BrionPos} {\sc Michel Brion:}
  Positivity in the Grothendieck group of complex flag varieties,
  J. Algebra (special volume in honor of Claudio Procesi) 258 (2002), 137--159.
  \url{http://arxiv.org/abs/math/0105254}

\bibitem[EKR61]{EKR} {\sc P\'al Erd\H os, Chao Ko, Richard Rado:}
  Intersection theorems for systems of finite sets,
  Quarterly Journal of Mathematics, Oxford Series, 
  series 2 (1961) 12: 313--320. 
  \url{http://dx.doi.org/10.1093/qmath/12.1.313}

\bibitem[FoS]{FS} {\sc Nicolas Ford, David Speyer:} in preparation.

\bibitem[Fr87]{Frankl} {\sc Peter Frankl:}
  The shifting technique in extremal set theory.
  Surveys in combinatorics 1987 (New Cross, 1987), 81–110, 
  London Math. Soc. Lecture Note Ser., 123, 
  Cambridge Univ. Press, Cambridge, 1987. 

\bibitem[Fu92]{Fulton92} {\sc William Fulton:}
  Flags, Schubert polynomials, degeneracy loci,
  and determinantal formulas,
  Duke Math. J. 65 (1992), no.~3, 381--420.

\bibitem[Gr00]{Graham} {\sc William Graham:}
  Positivity in equivariant Schubert calculus,
  Duke Math. J. 109 (2001), no. 3, 599--614.
  \url{http://arxiv.org/abs/math.AG/9908172}

\bibitem[HL]{HeLam} {\sc Xuhua He, Thomas Lam:}
  Projected Richardson varieties and affine Schubert varieties,
  preprint.
  \url{http://arxiv.org/abs/1106.2586}

\bibitem[Ka02]{Kalai} {\sc Gil Kalai:}
  Algebraic shifting.
  Computational commutative algebra and combinatorics (Osaka, 1999), 121--163, 
  Adv. Stud. Pure Math., 33, Math. Soc. Japan, Tokyo, 2002. 

\bibitem[KnMY09]{KMY} {\sc Allen Knutson, Ezra Miller, Alex Yong:} 
  Gr\"obner geometry of vertex decompositions and of flagged tableaux. 
  J. Reine Angew. Math. 630 (2009), 1--31. \\
  \url{http://arxiv.org/abs/math/0502144}

\bibitem[KnTao03]{KT} {\sc \bysame, Terence Tao:}
  Puzzles and (equivariant) cohomology of Grassmannians,
  Duke Math. J.  119  (2003),  no. 2, 221--260.
  \url{http://arxiv.org/abs/math.AT/0112150}

\junk{
\bibitem[KnYong]{KY} {\sc \bysame, Alex Yong:} 
  A formula for K-theory truncation Schubert calculus.
  IMRN 2004 (70) 3741--3756.
}

\bibitem[K]{unpub} {\sc \bysame:}
  Puzzles, positroid varieties, and equivariant $K$-theory of Grassmannians,
  unpublished.
  \url{http://arxiv.org/abs/1008.4302}
  

\bibitem[KLS13]{KLS} {\sc \bysame, T. Lam, D. Speyer:}
  Positroid varieties: juggling and geometry,
  Compositio Mathematica, Volume 149, Issue 10, October 2013, pp 1710--1752.
  \url{http://arxiv.org/abs/1111.3660}

\bibitem[KLS14]{KLS2} {\sc \bysame, Thomas Lam, David Speyer:}
  Projections of Richardson varieties, 
  Journal f\"ur die reine und angewandte Mathematik (Crelle's Journal)
  2014.687 (2014): 133--157.
  \url{http://arxiv.org/abs/1008.3939}

\bibitem[KnLed]{KL} {\sc \bysame, Mathias Lederer:} 
  A $K_T$-deformation of the ring of symmetric functions, in preparation.

\bibitem[LS07]{LSLittle} {\sc Thomas Lam, Mark Shimozono:}
  A Little bijection for affine Stanley symmetric functions.
  S\'em. Lothar. Combin. 54A (2005/07), Art. B54Ai, 12 pp. 
  \url{http://arxiv.org/abs/math/0601483}

\bibitem[MPP]{MP} {\sc Kar\'ola M\'esz\'aros, Greta Panova, Alex Postnikov:}
  Schur times Schubert via the Fomin-Kirillov algebra, 
  Electronic J. Combinatorics, Vol 21, Issue 1 (2014). \\
  \url{http://arxiv.org/abs/1210.1295}

\bibitem[Pos]{Postnikov} {\sc Alex Postnikov:}
  Total positivity, Grassmannians, and networks, preprint. \\
  \url{http://www-math.mit.edu/~apost/papers/tpgrass.pdf}

\bibitem[Va06]{Vakil} {\sc Ravi Vakil:}
  A geometric Littlewood-Richardson rule,
  Annals of Math. 164 (2006), 371--422. 
  \url{http://annals.math.princeton.edu/annals/2006/164-2/p01.xhtml}

\bibitem[W75]{White} {\sc Neil White:} 
  The bracket ring of a combinatorial geometry. I. 
  Trans. Amer. Math. Soc. 202 (1975), 79--95. 

\bibitem[ZJ09]{PZJ} {\sc Paul Zinn-Justin:}
  Littlewood-Richardson coefficients and integrable tilings. 
  Electron. J. Combin. 16 (2009), Research Paper 12.
  \url{http://arxiv.org/abs/0809.2392}
\end{thebibliography}
\end{document}